\documentclass{amsart}

\usepackage{mathtools}
\usepackage{enumitem}

\usepackage{todonotes}

\usepackage{xcolor}
\usepackage{graphics}
\usepackage{subcaption}

\newtheorem{theorem}{Theorem}[section]
\newtheorem{lemma}[theorem]{Lemma}
\newtheorem{corollary}[theorem]{Corollary}

\theoremstyle{definition}
\newtheorem{definition}[theorem]{Definition}
\newtheorem{example}[theorem]{Example}

\theoremstyle{remark}
\newtheorem{remark}[theorem]{Remark}

\numberwithin{equation}{section}

\newcommand{\N}{\mathbb{N}}

\newcommand{\R}{\mathbb{R}}

\newcommand{\Rn}{\R^n}

\newcommand{\tr}{\operatorname{tr}}
\newcommand{\sgn}{\operatorname{sgn}}

\renewcommand{\div}{\operatorname{div}}

\newcommand{\hlimsup}{\operatornamewithlimits{limsup^\ast}}
\newcommand{\hliminf}{\operatornamewithlimits{liminf_\ast}}

\begin{document}
	\title{Viscosity Solutions for Doubly-Nonlinear Evolution Equations}
	
	\author{Luca Courte}
	\address[Luca Courte]{Abteilung f\"ur Angewandte Mathematik,
		Albert-Ludwigs-Universit\"at Freiburg, Raum 228, Hermann-Herder-Straße 10, 79104 Freiburg i. Br.}
	\email{luca.courte@mathematik.uni-freiburg.de}
	\urladdr{https://aam.uni-freiburg.de/mitarb/courte/index.html}
	
	\author{Patrick Dondl}
	\address[Patrick Dondl]{Abteilung f\"ur Angewandte Mathematik,
		Albert-Ludwigs-Universit\"at Freiburg, Raum 217, Hermann-Herder-Straße 10, 79104 Freiburg i. Br.}
	\email{patrick.dondl@mathematik.uni-freiburg.de}
	\urladdr{https://aam.uni-freiburg.de/agdo/index.html}
	
	\subjclass[2010]{35D40, 35G31, 35K55, 34E15, 34A60, 34C55}
	
	\keywords{viscosity solution, partial differential inclusions, doubly nonlinear equations, vanishing viscosity limit, minimizing movement, hysteresis}
	
	\date{\today}
	
	\begin{abstract}
		We extend the theory of viscosity solutions to treat scalar-valued doubly-nonlinear evolution equations.  Such equations arise naturally in many mechanical models including a dry friction. After providing a suitable definition for discontinuous viscosity solutions in this setting, we show that Perron's construction is still available, i.e., we prove an existence result. Moreover, we will prove comparison principles and stability results for these problems. The theoretical considerations are accompanied by several examples, e.g., we prove the existence of a solution to a rate-independent level-set mean curvature flow. Finally, we discuss in detail a rate-independent ordinary differential equation stemming from a problem with non-convex energy. We show that the solution obtained by maximal minimizing movements and the solution obtained by the vanishing viscosity method coincide with the upper and lower Perron solutions and show the emergence of a rate-independent hysteresis loop.
	\end{abstract}
	
	\maketitle
	
	\section{Introduction}\label{sec:intro}
	In this article, we consider scalar-valued doubly-nonlinear evolution equations and we discuss the existence and uniqueness of solutions from a viscosity solution perspective \cite{Crandall1992, CourteDondlStefanelli}. To be precise, given an open domain $\Omega \subset \Rn$, a time interval $I = (0, T)$ with $T >0$, and a boundary condition $u_0 : \partial_P (\Omega\times I) \to \R$, where $\partial_P (\Omega\times I)$ is the parabolic boundary, i.e., $\partial_P (\Omega\times I) \coloneqq \Omega \times \{0\} \cup \partial \Omega \times I$, we are interested in the existence and properties of \emph{discontinuous viscosity solutions} $u : \Omega \times I \to \R$ to differential inclusions of the following type
	\begin{align}
	F(x, t, u, u_t, \nabla u, D^2 u) \in \mathcal{S}(u_t) G(x, t, u, \nabla u) &\text{ in } \Omega \times I, \label{eq:pde}\\
	u(\cdot, 0) = u_0 &\text{ on } \partial_P(\Omega\times I), \label{eq:pde_bdry}
	\end{align}
	where $F : \Omega \times I \times \R \times \R \times \Rn \times \operatorname{Sym}(n) \to \R$ is a possibly nonlinear function, $\operatorname{Sym}(n)$ are the symmetric $n\times n$ matrices, $G:\Omega \times I \times \R \times \Rn \to [0, \infty)$ is a positive and possibly nonlinear function, and $\mathcal{S} : \R \to \mathcal{P}(\R) \setminus \{\emptyset\}$ is a set-valued function. The symbol $\mathcal{P}(\R)$ denotes the power-set of $\R$. 
	
	Such doubly nonlinear set-valued equations arise naturally while modeling dry friction (or rate independent) mechanics. For an overview, see \cite{Mielke2015}. The prevalence of rate independent friction is due to the fact that it (as well as other, more complex dissipation models) emerges as the limit of viscous evolution in highly oscillatory (wiggly) environments \cite{Mielke2011, MR4027693}. The concept of energetic solutions derived using iterative minimization schemes has been shown to be powerful method to treat such rate independent models \cite{Francfort2006}. Classical applications include models for fracture \cite{Francfort1998} and phase transitions \cite{MielkeTheil}. More recently, the concept has been applied to damage mechanics \cite{knees_rossi_zanini_2019, Knees2019, knees_17} and plasticity \cite{Mielke2004, Mielke2017}, among many others (e.g., \cite{Mielke2007, Rossi2008, Mielke2012, Mielke2016, Knees2019}). One central issue when considering iterative minimization as a means to prove existence of solutions for rate independent systems is that a certain amount of compactness is necessary -- in all but some exceptional, linear cases, the dissipation distance must be compact on sets where the potential energy is finite. Viscosity solutions provide a means to overcome this issue, albeit at the expense that only scalar problems can be treated. 
	
	The concept of discontinuous viscosity solutions \cite{Barles1987, Bardi1997, Crandall1992, Barles2013} turns out to be a good setting for  the study of \eqref{eq:pde}, as potential time discontinuities, i.e., jumps, are one of the main features of doubly-nonlinear evolution equations.  The proofs of comparison principles and existence of solutions follow the classical theory and are therefore relegated in Appendix \ref{app:proof}. Instead, we focus on the usefulness of this concept by applying the results to several examples relevant in mechanics. In particular, in section \ref{sec:per_sol}, we consider an ordinary differential inclusion and show how the Perron solutions relate to the solutions obtained by the vanishing viscosity limit and the minimizing movements procedure (see, e.g., \cite{Efendiev_Mielke, MR3021776}). The extension of the theory of viscosity solutions to this novel setting and the relationship between the Perron solutions, the vanishing viscosity solution, and the maximal minimizing movement solution are the main contributions of this article.
	
	The remainder of this article is organized as follows. In section \ref{sec:disc_vs}, we extend the notion of discontinuous viscosity solutions to equations of type \eqref{eq:pde} and present applications of the theory. As the proofs are similar to the case of (discontinuous) viscosity solutions in the sense of \cite{Crandall1992} they have been moved to Appendix \ref{app:proof} and the main focus in that section lies in the application of the results. With the help of Perron's construction \cite{Ishii1987}, we prove in subsection \ref{ssec:existence} an existence result for \eqref{eq:pde} under some mild continuity conditions on $G$ and $\mathcal{S}$. Contrary to the references mentioned in the beginning, this allows us to easily prove existence of solutions to equations, even if $-\mathcal{S}$ is not maximal monotone. Indeed, we highlight this by proving existence of solutions to the equation $$\mathcal{F}(u_t) - \Delta u \ni f(x, t),$$
	where $\mathcal{F} : \R \to \mathcal{P}(\R)$ models a combination of viscous friction, dry friction, and static friction (i.e., a variation of ``sticktion'' \cite{Visintin1994}), e.g., $\mathcal{F}(0) = [-2, 2]$ and $\mathcal{F}(a) = \{\sgn(a) + a\}$, whenever $a\not= 0$, see Example \ref{ex:fric}.

	In subsection \ref{ssec:comp}, we prove that equation \eqref{eq:pde} satisfies a strict comparison principle under some monotonicity assumptions, i.e., we can compare a strict subsolution with a supersolution and a subsolution with strict supersolution. We illustrate the usefulness of this in two examples. First, we will show the existence and uniqueness of a solution to a doubly nonlinear evolution equation with a degenerate elliptic operator, see Example \ref{ex:nonlinear}. Second, in Example \ref{ex:mcf}, we prove that there exists a family of level-set functions $u(\cdot,t) : \R^n \to \R$ satisfying
	\begin{align*}
	|\nabla u| \partial R(u_t) - |\nabla u|  \operatorname{div}(\tfrac{\nabla u}{|\nabla u|}) \ni f \text{ on } \{u(\cdot, t) = 0\},
	\end{align*}
	i.e., the set $\{u=0\}$ is a solution to a rate-independent level-set mean curvature flow if we assume $\partial R$ to be the subdifferential of the absolute value, \begin{equation} \label{eq:R}
	\partial R(a) \coloneqq \left\{ \begin{array}{ll} \lbrack-1, 1\rbrack &\text{ if } a = 0, \\ \sgn(a) &\text{ elsewhere.} \end{array} \right.
	\end{equation}
	Note that these techniques also work if we replace $\partial R$ by $\mathcal{F}$ as in Example \ref{ex:fric} or other (continuous) multi-valued maps. This provides a theory that extends beyond the advancements made in \cite{Visintin1990}, as we can treat problems with mixed dynamics.
	
	Finally, we devote section \ref{sec:per_sol} to study how discontinuous viscosity solutions of the rate-independent ordinary differential equation 
	\begin{align*}
	\partial R(u_t) + e(u) &\ni f,
	\end{align*}
	relate to other solution concepts when $\partial R$ is given by \eqref{eq:R}, $e : \R \to \R$ is non-monotone, and $f : \R \to \R$ is a given loading. We also explain how a rate-independent hysteresis loop emerges if $f$ is a periodic loading, see subsection \ref{ssec:hysteresis}. Let us briefly present the results obtained in this section. Given a subsolution $u$ and a supersolution $v$, the Perron solutions are the following two functions
	\begin{align}
	\label{eq:greater_perron}U(x, t) &\coloneqq \sup\{ \tilde{u}(x, t) \;|\; \tilde{u} \text{ is a subsolution with } u \le \tilde{u} \le v \}, \\
	\label{eq:smaller_perron}V(x, t) &\coloneqq \inf\{ \tilde{v}(x, t) \;|\; \tilde{v} \text{ is a supersolution with } u \le \tilde{v} \le v \},
	\end{align}
	and it turns out that they are discontinuous viscosity solutions, see Theorem \ref{thm:perron}. Assume $f$ is increasing, then the smaller Perron solution \eqref{eq:smaller_perron} corresponds to the solution that arises during a vanishing viscosity limit. On the other hand, it turns out that the greater Perron solution \eqref{eq:greater_perron} can be constructed using a maximal minimizing movements sequence, i.e., it is an energetic solution. As any other discontinuous viscosity solution lies between $V$ and $U$, this shows that this theory provides solutions that capture the behavior of any (meaningful) evolution to this equation. We conjecture that this link persists even in the case of the partial differential equation \eqref{eq:pde} if all the involved solution concepts are meaningful.
				
	\section{Discontinuous Viscosity Solutions}\label{sec:disc_vs}
	For the convenience of the reader, we start by providing the basic concepts and notations required to state the definition of discontinuous viscosity solutions. Let $\Omega \subset \Rn$, we write $USC(\Omega)$ for the set of all upper-semicontinuous functions from $\Omega$ to $\R$. Likewise, $LSC(\Omega)$ contains all lower-semicontinuous functions from $\Omega$ to $\R$. For upper- and lower-semicontinuous functions, one can introduce the notion of super- and sub-jets respectively \cite[Section~8]{Crandall1992}.
	
	\begin{definition}[Parabolic Jets]
		Let $\Omega \subset \Rn$ be an open domain, $T > 0$, $I \coloneqq (0, T)$, and $u \in USC(\Omega \times I)$. The second order parabolic super-jet $\mathcal{P}^{2, +}u(\hat x, \hat t)$ at $(\hat x, \hat t) \in \Omega \times I$ contains all triplets $(a, p, X) \in \R \times \Rn \times \operatorname{Sym}(n)$ such that
		\[
			u(x, t) \le u(\hat x, \hat t) + a(t-\hat t) + \left<p, x - \hat x\right> + \left<X(x-\hat x), x-\hat x\right> + {\rm o}(|t-\hat t| + |x-\hat x|^2).
		\]
		Moreover, the second order parabolic sub-jet of a lower-semicontinuous function $v \in LSC(\Omega \times I)$ is defined as $\mathcal{P}^{2, -} v(\hat y, \hat s) \coloneqq - \mathcal{P}^{2, +} (-v(\hat y, \hat s))$ for all $(y, s) \in \Omega \times I$, i.e., it contains all triplets $(b, q, Y) \in \R \times \Rn \times \operatorname{Sym}(n)$ such that
		\[
		v(y, s) \ge v(\hat y, \hat s) + b(s-\hat s) + \left<q, y - \hat y\right> + \left<Y(y-\hat y), y-\hat y\right> + {\rm o}(|s-\hat s| + |y-\hat y|^2).
		\]
	\end{definition}
		
	\begin{definition}[Semicontinuous Envelopes]
		Let $\Omega \subset \Rn$ and $f : \Omega \to \R$ a function, we call the function
		\[
			f^*(x) \coloneqq \inf\{ g \in USC(\Omega) \;|\; g \ge f \}
		\]
		the upper-semicontinuous envelope of $f$. Analogously,
		\[
			f_*(x) \coloneqq \sup\{ g \in LSC(\Omega) \;|\; g \le f \}
		\]
		is the lower-semicontinuous envelope of $f$.
	\end{definition}

	Now that we have all the ingredients in hand, we can define discontinuous viscosity solutions in our novel setting. The following definition extends the definition in \cite{CourteDondlStefanelli}. For a detailed discussion of viscosity solutions, we refer to \cite{Barles1987, Bardi1997, Crandall1992, Barles2013}.

	\begin{definition}[Discontinuous Viscosity Solutions]\label{def:dis_vc}
		Let $\Omega \subset \Rn$ be an open domain and $I\coloneqq (0, T)$ a finite time interval. We call an upper-semicontinuous function $u \in USC(\Omega \times I)$ a viscosity subsolution of \eqref{eq:pde} if for every $(x, t) \in \Omega \times I$ and every $(a, p, X) \in \mathcal{P}^{2, +} u(x, t)$ and some $\mu \in \mathcal{S}(a)$ we have
		\begin{equation}\label{eq:def_subsol}
			F_*(x, t, u(x, t), a, p, X) \le \mu G(x, t, u(x, t), p).
		\end{equation}
		Moreover, a lower-semicontinuous function $v \in LSC(\Omega \times I)$ is called a viscosity supersolution of \eqref{eq:pde} if for every $(x, t) \in \Omega \times I$ and every $(b, q, Y) \in \mathcal{P}^{2, -} v(x, t)$ and some $\nu \in \mathcal{S}(b)$ we have
		\begin{equation}\label{eq:def_supersol}
			F^*(x, t, v(x, t), b, q, Y) \ge \nu G(x, t, v(x, t), q).
		\end{equation}
		Finally, a function $u$ whose upper-semicontinuous envelope $u^*$ is a viscosity subsolution and whose lower-semicontinuous envelope $u_*$ is a viscosity supersolution, is called a discontinuous viscosity solution. If a discontinuous viscosity solution $u$ is continuous, i.e., $u^* = u_* = u$ then $u$ is simply called viscosity solution.
	\end{definition}

\begin{remark}
	We only obtain a meaningful solution concept if $F$ is degenerate elliptic, i.e., $F(x, t, r, a, p, Y) \le F(x, t, r, a, p, X)$ if $X\le Y$ for all $(x, t) \in \Omega\times I$, $r, a \in \R$, $p \in \Rn$ and $X \in \operatorname{Sym}(n)$ and $G$ continuous.
\end{remark}

\subsection{Existence}\label{ssec:existence}
We will employ Perron's construction (see \cite{Ishii1987}) to prove the existence of discontinuous viscosity solutions of equation \eqref{eq:pde}. Perron's method is based on the (semi-)continuity of the equation. Due to the definition of discontinuous sub- and supersolutions, no further assumptions on $F$ are necessary here. Furthermore, we assume that $G$ is continuous and that $\mathcal{S}$ satisfies the following conditions:
\begin{enumerate}[label=C\arabic*)]
	\item\label{it:C_seq} If $a_n \to a$ then any sequence $\mu_n \in \mathcal{S}(a_n)$ has a subsequence $\mu_{n_k}$ such that $\mu_{n_k} \to \mu \in \mathcal{S}(a)$.
	\item\label{it:C_epsDelta} For all $a\in \R$ and for each $\epsilon > 0$ there is a $\delta > 0$ such that for all $b\in\R$ with $|a-b| < \delta$ there are $\mu \in \mathcal{S}(a)$ and $\nu \in \mathcal{S}(b)$ with $|\mu-\nu| <\epsilon$,
	\item\label{it:C_compact} For all $a\in \R$ the set $\mathcal{S}(a)$ is compact.
\end{enumerate}
In \cite{CourteDondlStefanelli} it was already noted that if $-\mathcal{S}$ is maximal monotone, these conditions are satisfied. The arguments required to prove the following result are an adaptation of the theory from \cite[22-24]{Crandall1992}, \cite[302-305]{Bardi1997}.

\begin{theorem}[Perron's Method]\label{thm:perron}
	Let $\Omega \subset \Rn$, $I \coloneqq (0, T)$, $T > 0$. Moreover, let $u$ be a subsolution and $v$ be a supersolution of \eqref{eq:pde} with $u \le v$ on $\Omega \times I$. Then the functions
	\begin{align*}
	U(x, t) &\coloneqq \sup\{ \tilde{u}(x, t) \;|\; \tilde{u} \text{ is a subsolution with } u \le \tilde{u} \le v \} \\
	V(x, t) &\coloneqq \inf\{ \tilde{v}(x, t) \;|\; \tilde{v} \text{ is a supersolution with } u \le \tilde{v} \le v \}
	\end{align*}
	are discontinuous viscosity solutions of \eqref{eq:pde}. Moreover, if $u$ and $v$ are continuous, it holds $V^* \le U$, $V \le U_*$, and in particular
	\[
		u \le V \le U \le v ~~\text{ on $\Omega \times I$.}
	\]
\end{theorem}
\begin{proof}
	This theorem is proved in \ref{proof:perron}.
\end{proof}

\begin{remark}
	As soon as comparison holds for the equation, it follows that $U^* \le U_*$ and $V^* \le V_*$, i.e. $U$ and $V$ are continuous. Moreover, we obtain $U = V$ if $u = v$ on the parabolic boundary.
\end{remark}

\begin{example}[``Sticktion'']\label{ex:fric}
	Let $\Omega \subset \Rn$ be an open domain and $I =(0, T)$ a finite time interval. We consider the following differential inclusion
	\begin{align*}
		\mathcal{F}(u_t) - \Delta u &\ni f  \text{ in } \Omega \times I, \\
		u(\cdot, 0) &= 0 \text{ on } \Omega,
	\end{align*}
	where $\mathcal{F} : \R \to \mathcal{P}(\R)$ is given by $\mathcal{F}(0) \coloneqq [-2, 2]$ and $\mathcal{F}(a) \coloneqq \{\sgn(a) + a\}$, whenever $a\not= 0$. This set-valued function can be seen as a model for a dissipation that combines viscous friction, dry friction, and static friction. Moreover, the right-hand side $f \in C(\Omega \times I)$ is assumed to be a bounded continuous loading.
	
	In order to employ Perron's construction, we have to verify that $\mathcal{F}$ satisfies the required conditions. First of all, by definition $\mathcal{F}$ satisfies \ref{it:C_compact}. To simplify notation, define $F : \R \setminus \{0\} \to \R$, $F(a) \coloneqq \sgn(a) + a$. If $a\not=0$ then $\mu \in \mathcal{F}(a)$ if and only if $\mu = F(a)$. Now, take any sequence $a_n \to a$ and any sequence $\mu_n \in \mathcal{F}(a_n)$. Due to the boundedness of $a_n$, the sequence $\mu_n$ is bounded and we can extract a converging subsequence. If $a \not= 0$ then we are at a point of continuity of $F$ and hence the subsequence converges to $F(a)$, i.e., to an element of $\mathcal{F}(a)$. Now assume that $a=0$, then we can pass to a further subsequence such that either all elements $a_n = 0$ or all $a_n \not= 0$. In the first case, the compactness of $\mathcal{F}(0)$ shows that a subsequence of $\mu_n$ converges to an element of $\mathcal{F}(0)$. In the second case, we pass to a further subsequence such that $a_n > 0$ or $a_n < 0$. In either case, we can look at the right-continuous or left-continuous extension of $F$ and see that $\mu_n \to \mu \in \mathcal{F}(0)$. This finishes the proof that $\mathcal{F}$ satisfies \ref{it:C_seq}. With a similar argument one can also show that $\mathcal{F}$ satisfies \ref{it:C_epsDelta}.
	
	Hence, we showed that all the prerequisites to apply Theorem \ref{thm:perron} are fulfilled. The missing ingredients are the existence of a sub- and a supersolution that satisfy the Dirichlet boundary condition in a strong sense. Let $\underline{u}(x, t) \coloneqq -||f||_\infty t$, then for all $(x, t) \in \Omega \times I$ it holds
	\[
		\mathcal{F}(\underline{u}_t(x, t)) - \Delta \underline{u}(x, t) \ni -1 - ||f||_\infty \le f(x, t).
	\]
	Therefore, $\underline{u}$ is a subsolution and we can prove analogously that $\overline{u} \coloneqq -\underline{u}$ is a supersolution. Moreover, it holds that $\underline{u} \le \overline{u}$ and $\underline{u}(\cdot, 0) = \overline{u}(\cdot, 0) = 0$. Hence, we can construct a discontinuous viscosity solution using Perron's method.
\end{example}

\subsection{Comparison Principles}\label{ssec:comp}

In the previous subsection, we saw how Perron's construction can be used to construct discontinuous viscosity solutions. Unfortunately, in general, one cannot hope to find a unique solution to equation \eqref{eq:pde}. Hence, we cannot hope that \eqref{eq:pde} fulfills a strong comparison principle. Despite this, it turns out that under some mild monotonicity assumptions, we can still prove a comparison principle with a strict subsolution and a supersolution or a subsolution and a strict supersolution. This strict comparison principle allows us to study the behavior of any solution to the equation. This turns out to be particularly useful when studying the Perron solution in section \ref{sec:per_sol}.  Note, that we will provide a version of the comparison principle that is suitable to treat geometric singularities as they arise in the mean curvature flow.

As already mentioned we have to make some additional assumptions on $F$, $G$ and $\mathcal{S}$. Indeed, we assume that the function $G :  \Omega \times I \times \R \times \Rn \to [0, \infty)$ satisfies
\begin{enumerate}[label=G\arabic*)]
	\item\label{it:G_Lipschitz} There is a constant $L_G > 0$ such that
	$\left|G(x, t, r, p)-G(x, t, s, p)\right| \le L_G |r-s|$ for all $r, s \in \R$, $(x, t) \in \Omega \times I$ and $p \in \Rn$,
	\item\label{it:G_modulus} There exists a modulus of continuity $\omega_G$ such that for all $\alpha > 0$ big enough, it holds
	\begin{align*}
	&G(x, t, r, 4\alpha |x-y|^2 (x-y)) - G(y, t, r, 4\alpha |x-y|^2 (x-y)) &&\\&\le \omega_G(|x-y| + \alpha |x-y|^4),
	\end{align*}
	whenever $x, y \in \Omega, t \in I, r\in \R$.
\end{enumerate}
The set-valued function $\mathcal{S} :\R \to \mathcal{P}(\R) \setminus \{\emptyset\}$ is decreasing and bounded, i.e.,
\begin{enumerate}[label=S\arabic*)]
	\item\label{it:S_mon} $\sup S(a) \le \inf S(b)$ for all $a, b \in \R$ with $a > b$,
	\item\label{it:S_bdd} It holds $\mathcal{S}_\mathrm{max} \coloneqq \sup \left\{\bigcup_{a\in \R} \bigcup_{s\in\mathcal{S}(a)} |s|\right\} < \infty$ .
\end{enumerate}
Moreover, $F : \Omega \times I \times \R \times \R \times \Rn \times \mathrm{Sym}(n) \to \R$ satisfies
\begin{enumerate}[label=F\arabic*)]
	\item\label{it:F_inc}$\gamma(r-s) \le F_*(x, t, r, a, p, X) - F_*(x, t, s, a, p, X)$ for all $r, s \in \R$ with $r \ge s$ and $\gamma \coloneqq \mathcal{S}_{\textrm{max}}L_G$ with $L_G$ from \ref{it:G_Lipschitz} and $\mathcal{S}_{\mathrm{max}}$ from \ref{it:S_bdd},
		\item\label{it:F_modulus} There exists a modulus of continuity $\omega_F$ such that
	\begin{align*}
	F^*(y, t, r, b, p, Y) - F_*(x, t, r, a, p, X) \le \omega_F(|x-y| + \alpha |x-y|^4),
	\end{align*}
	whenever $x, y \in \Omega, t \in I, r, a \in \R, \alpha \in \R_{\ge0}$, $p \coloneqq 4\alpha |x-y|^2 (x-y)$, $b \le a$, and 
	\[
	-4||Z|| \left(\begin{array}{cc}\operatorname{Id} & 0 \\ 0 & \operatorname{Id}\end{array}\right) \le \left(\begin{array}{cc}X & 0 \\ 0 & -Y\end{array}\right)\le \left(\begin{array}{cc}Z+\frac{1}{2||Z||}Z^2 & -(Z+\frac{1}{2||Z||}Z^2) \\ -(Z+\frac{1}{2||Z||}Z^2 ) & Z+\frac{1}{2||Z||}Z^2 \end{array}\right),\]
	with $Z \coloneqq 4 \alpha |x-y|^2 \operatorname{Id} + 8\alpha (x-y) \otimes (x-y)$. In particular, the last inequality implies that $X \le Y$ and $||X||, ||Y|| \le C |x-y|^2$.
\end{enumerate}

\begin{theorem}[Strict Comparison Principle on Bounded Domains]\label{thm:comp}
	Let $\Omega \subset\Rn$ be a bounded, open domain and $I \coloneqq (0, T)$  a finite time interval. Moreover, let $u$ be a strict viscosity subsolution, i.e., there exists $\lambda > 0$ such that for every $(x, t) \in \Omega \times I$ and every $(a, p, X) \in \mathcal{P}^{2, +} u(x, t)$ and some $\mu \in \mathcal{S}(a)$ we have
	   \begin{equation}
	   F_*(x, t, u(x, t), a, p, X) \le \mu G(x, t, u(x, t), p) - \lambda,
	   \end{equation}
	and $v$ be a viscosity supersolution of \eqref{eq:pde} with $u \le v$ on the parabolic boundary $\partial_P (\Omega \times I)$, i.e., on $ \partial\Omega \times I \cup \Omega \times \{0\}$, then $u \le v$ in $\Omega \times I$.
	
	The same result holds if $u$ is merely a subsolution and $v$ is a strict supersolution.
\end{theorem}

\begin{proof}
	This theorem is proved in \ref{proof:comp}.
\end{proof}

\begin{remark}\label{rmk:comp}
	If the partial differential inclusion is well-behaved, the strict comparison principle can be extended to a comparison principle, i.e., we can compare any subsolution with any supersolution. For instance, consider the equation
	\[
		u_t + F(x, t, u, \nabla u, D^2 u) \in \mathcal{S}(u_t) G(x, t, u, \nabla u),
	\]
	such that \ref{it:F_modulus}, \ref{it:G_Lipschitz}, \ref{it:G_modulus}, \ref{it:S_bdd}, and \ref{it:S_mon} hold, and $F$ is either increasing or Lipschitz in the third variable. Then an exponential rescaling, i.e., $U(x, t) \coloneqq e^{-\lambda t} u(x, t)$ can be used to cast the equation and any subsolution and any supersolution in a form such that Theorem \ref{thm:comp} is applicable.
	
	Note that whenever comparison holds for any subsolution and any supersolution, the Perron solutions from Theorem \ref{thm:perron} coincide. 
\end{remark}

\begin{example}[Equation with degenerate elliptic operator]\label{ex:nonlinear}
	Now that we have proved a comparison principle, we can present the existence theory we developed in full glory. For this, we give an example of a doubly-nonlinear equation with a degenerate elliptic operator.
	
	Assume that $\Omega \coloneqq B_1(0) \subset \R^2$, and $I = (0, T)$ a finite time-interval and consider the following equation
	\begin{align*}
	\partial R(u_t) + \max\{- \Delta u, 0 \} + u^3 \ni f &, \text{ in } \Omega \times [0, T) \\
	u = 0 &, \text{ on }  \partial_P(\Omega \times I) ,
	\end{align*}
	with $f \in C^0(\Omega \times [0, T])$ and $\partial R$ the subdifferential of the absolute value, see \eqref{eq:R}.
	
	Moreover, we assume that there is a continuous, increasing function $m : I \to [0, \infty)$ with $m(0) = 0$ such that 
	\[
	|f(x, t)| \le 1 + \left(m(t)(1-|x|^2)\right)^3.
	\]
	
Let $\underline{u}(x, t) \coloneqq m(t)(|x|^2-1)$ and $\overline{u}(x, t) \coloneqq m(t)(1-|x|^2)$, then it holds
\begin{align*}
\partial R(\underline{u}_t) + \max\{ -\Delta \underline{u}, 0 \}+ \underline{u}^3 \ni -1 + 0 - \left(m(t)(1-|x|^2)\right)^3
 \le f(x, t)
\end{align*}
and likewise
\[
\partial R(\overline{u}_t) + \max\{ -\Delta \overline{u}, 0 \}+ \overline{u}^3 \ge  f(x, t).
\]

Hence, $\underline{u}$ and $\overline{u}$ are viscosity sub- and supersolutions with $\underline{u} = \overline{u} = 0$ on the parabolic boundary. Perron's method, Theorem \ref{thm:perron}, now allows to construct a viscosity solution to the problem. The uniqueness follows by comparison (see Theorem \ref{thm:comp} and Remark \ref{rmk:comp}), as the equation is strictly monotone in $u$.
\end{example}

We can also extend the strict comparison principle on the whole domain $\Rn$. To be able to prove such a result, the functions $F$ and $G$ have to satisfy stronger conditions. To be precise, we use the following conditions:
\begin{enumerate}[label=FU)]
	\item\label{it:F_growth} The function $F$ can be written as $F(x, t, r, a, p, X) = F_1(x, t, r, a) + F_2(t, p, X)$ with $F_1$ continuous. Moreover, there are $C_{F_1} > 0$, $K_F > 0$, and moduli of continuity $\omega_{F_1}$ and $\omega_{F_2}$, such that for all $x, y \in \Rn$, $t \in I$, $r\in \R$, $a, b \in \Rn$, $p, q \in \Rn$, and $X, Y \in \operatorname{Sym}(n)$, the following conditions hold:
	\begin{enumerate}[label=(\roman*)]
		\item $F_1(y, t, r, b) - F_1(x, t, r, a) \le \omega_{F_1}(|x-y| + b-a)$, whenever $a \le b$,
		\item $F_2^*(t, p, Y) - {F_2}_*(t, p, X) \le \omega_{F_2}(|x-y| + \alpha |x-y|^4)$, 
		whenever $p \coloneqq 4\alpha |x-y|^2 (x-y)$, and 
		\[
		-4||Z|| \left(\begin{array}{cc}\operatorname{Id} & 0 \\ 0 & \operatorname{Id}\end{array}\right) \le \left(\begin{array}{cc}X & 0 \\ 0 & -Y\end{array}\right)\le \left(\begin{array}{cc}Z+\frac{1}{2||Z||}Z^2 & -(Z+\frac{1}{2||Z||}Z^2) \\ -(Z+\frac{1}{2||Z||}Z^2 ) & Z+\frac{1}{2||Z||}Z^2 \end{array}\right),\]
		with $Z \coloneqq 4 \alpha |x-y|^2 \operatorname{Id} + 8\alpha (x-y) \otimes (x-y)$.
		\item $F_1(y, t, r, b) - F_1(x, t, r, a) \le C_{F_1} + K_F |x-y|$, whenever $b \le a$
		\item $F_2^*(t, p, Y) - {F_2}_*(t, p, X)$ is locally bounded.
	\end{enumerate}
	
\end{enumerate}
\begin{enumerate}[label=GU)]
	\item\label{it:G_growth} The function $G$ can be written as $G(x, t, r, p) = G_1(x, t, r) + G_2(t, p)$ with $G_1$ and $G_2$ continuous. Moreover, there are $C_{G_1} > 0$,  
	$K_G > 0$, and a modulus of continuity $\omega_{G_1}$, such that for all $x, y \in \Rn$, $t \in I$, $r\in \R$, and $p, q \in \Rn$, the following conditions hold:
	\begin{enumerate}[label=(\roman*)]
		\item $G_1(x, t, r) - G_1(y, t, r) \le \omega_{G_1}(|x-y|)$,
		\item $G_1(x, t, r) - G_1(y, t, r) \le C_{G_1} + K_G |x-y|$.
	\end{enumerate}
\end{enumerate}

\begin{theorem}[Comparison Principle in $\Rn$]\label{thm:comp_rn}
	Let $I \coloneqq (0, T)$ with $T < \infty$. Assume that $F$ satisfies \emph{\ref{it:F_inc}}, with $\gamma = \mathcal{S}_{\textrm{max}}L_G +\eta$, $\eta > 0$, \emph{\ref{it:F_growth}}, $G$ satisfies \emph{\ref{it:G_Lipschitz}}, \emph{\ref{it:G_growth}}, and $\mathcal{S}$ fulfills \emph{\ref{it:S_mon}}, \emph{\ref{it:S_bdd}} on $\Rn$.
	Moreover, let $u$ be a viscosity subsolution and $v$ be a viscosity supersolution of \eqref{eq:pde} with 
	\begin{equation}\label{eq:comp_rn_assumption}
	u(x, t) - v(y, t) \le L(1+|x|+|y|) \text{ for all } (x, y, t) \in \Rn \times \Rn \times I
	\end{equation}
	for some $L > 0$ which is independent of $t$. If $u(\cdot, 0) \le v(\cdot, 0)$ then $u \le v \text{ in } \Rn \times I$.
\end{theorem}

\begin{proof}
	This theorem is proved in \ref{proof:comp_rn}.
\end{proof}

\begin{remark}
	If the difference of the sub- and supersolution in the assumptions of Theorem \ref{thm:comp_rn} is a priori bounded, the growth estimate holds trivially. Hence, $\eta$ can be chosen as zero, i.e., we do not need to have strong monotonicity. However to deal with the second part of the theorem, one of the functions then has to be a strict sub- or supersolution as in Theorem \ref{thm:comp}.
\end{remark}

\begin{example}[Level-set Mean Curvature Flow]\label{ex:mcf}
	It turns out that viscosity solutions are particularly well-suited to prove existence of solutions to the level-set formulation of the mean curvature flow \cite{Crandall1992}. In fact, let $d$ be the signed distance function of some surface, then it turns out that the normal velocity is given by $d_t$ and the curvature by $\Delta d$. Replacing $d$ by any differentiable function $u$ with $\sgn(u) = \sgn(d)$ and $u = 0 \iff d = 0$, we can compute the normal velocity by $\frac{u_t}{|\nabla u|}$ and the curvature by $\operatorname{div}(\frac{\nabla u}{|\nabla u|})$. Therefore, we postulate the following equation (see also \eqref{eq:R}) as a rate-independent mean curvature flow
	\[
	\partial R(\tfrac{u_t}{|\nabla u|}) \ni \operatorname{div}(\tfrac{\nabla u}{|\nabla u|}) \text{ on } \{u(\cdot, t) = 0\}.
	\]
	Multiplying this equation by $|\nabla u|$, rearranging, and introducing a time-dependent forcing $f$, we obtain
	\begin{equation}\label{eq:ri_mcf}
	|\nabla u| \partial R(\tfrac{u_t}{|\nabla u|}) - |\nabla u|  \operatorname{div}(\tfrac{\nabla u}{|\nabla u|}) \ni f \text{ on } \{u(\cdot, t) = 0\}.
	\end{equation}
	
	Now to prove existence of a continuous function $u : \Rn \times I \to \R$ that solves \eqref{eq:ri_mcf}, we are going to solve the following equation
	\begin{equation}\label{eq:ri_mcf_monotone}
	|\nabla u| \partial R(u_t) - |\nabla u|  \operatorname{div}(\tfrac{\nabla u}{|\nabla u|}) + u \ni f \text{ on } \Rn \times I.
	\end{equation}
	Note, that in the set $\{u(\cdot, t) = 0\}$ the term $+u$ vanishes and it renders the whole equation monotone in $u$. Let us now introduce $$F(x, t, r, p, X) \coloneqq -\tr(X) + \tr(\frac{p\otimes p}{|p|^2}X) + r - f(x, t),\quad G(p) \coloneqq |p|,$$then \eqref{eq:ri_mcf_monotone} is equivalent to
	\begin{equation*}
	F(x, t, u, \nabla u, D^2 u) \in -\partial R(u_t) G(\nabla u) \text{ on } \Rn \times I.
	\end{equation*}
	The function $F$ is degenerate elliptic, discontinuous in $p=0$, $G$ is Lipschitz and can be decomposed as required by \ref{it:F_growth}, \ref{it:G_growth}. Indeed, the only difficulty is to prove that \ref{it:F_growth}(ii) is satisfied. For this, we need to compute the upper and lower semi-continuous envelopes. It is easy to see that
	\begin{align*}
	F^*(x, t, r, p, X) &= \left\{ \begin{array}{ll} F(x, t, r, p, X) &\text{ if } p \not= 0, \\ -\tr(X) + \lambda_\textrm{max}(X) + r - f(x, t) &\text{ if } p = 0, \end{array} \right. \\
	F_*(x, t, r, p, X) &= \left\{ \begin{array}{ll} F(x, t, r, p, X) &\text{ if } p \not= 0, \\ -\tr(X) + \lambda_\textrm{min}(X) + r - f(x, t) &\text{ if } p = 0, \end{array} \right.
	\end{align*}
	If $f$ is uniformly continuous, then \ref{it:F_growth}(ii) is satisfied by degenerate ellipticity if $p\not=0$. If $p = 0$, then $|x-y| = 0$ and hence $Z = 0$ which implies that $X = Y = 0$ and the inequality is trivially satisfied.
	
	\begin{theorem}\label{thm:mcf_sol}
		Assume that $f\in C(\Rn \times I)$ is uniformly continuous and it holds that $|f(x, t) -f(x, 0)| \le \Lambda t$, $u_0 \in C^2(\Rn)$ is such that there exists some $\mu \in [-1, 1]$ with
		\[
		\mu|\nabla u_0(x)| - |\nabla u_0(x)|  \operatorname{div}(\tfrac{\nabla u_0(x)}{|\nabla u_0(x)|}) + u_0(x) = f(x, 0),
		\]
		then there is a unique viscosity solution $u \in C(\Rn \times I)$ to \eqref{eq:ri_mcf_monotone} with $u(\cdot, 0) = u_0$.
	\end{theorem}
	\begin{proof}
		We already discussed that equation \eqref{eq:ri_mcf_monotone} satisfies conditions \ref{it:F_inc}, \ref{it:F_growth}, \ref{it:G_Lipschitz}, \ref{it:G_growth}, \ref{it:C_seq},  \ref{it:C_epsDelta}, \ref{it:S_bdd}, and \ref{it:S_mon}. Moreover, we have a strict monotonicity in $r$. Hence, the comparison principle, Theorem \ref{thm:comp_rn}, Remark \ref{rmk:comp}, and Perron's construction, Theorem \ref{thm:perron}, are applicable.
		
		As we assumed that the initial condition satisfies the equation, we can introduce
		$\overline{u}(x, t) \coloneqq u_0(x) + \Lambda t$, it holds
		\[
		|\nabla u_0| - |\nabla u_0|  \operatorname{div}(\tfrac{\nabla u_0}{|\nabla u_0|}) + u_0(x) + \Lambda t \ge f(x, 0) + \Lambda t \ge f(x, t).
		\]
		Hence, $\overline{u}$ is a viscosity supersolution with $\overline{u}(x, 0) = u_0(x)$. Likewise a viscosity subsolution $\underline{u}(x, t) \coloneqq u_0(x) - \Lambda t$ can be computed.
		
		Due to the existence of these solutions, we can construct a unique viscosity solution $u$ with $\underline{u} \le u \le \overline{u}$ and hence $u(\cdot, 0) = u_0$.
	\end{proof}
	It is a rather delicate affair to construct sub- and supersolutions that provide quantitative statements on the solutions of this mean-curvature flow model. Let us however mention how to pose a problem in this setting, consider an initial condition $u_0(x) \coloneqq \phi(\tilde{u}_{0}(x))$, where $\tilde{u}_0$ is some smooth level-set function and $\phi$ is a smooth, increasing cutoff function with $\phi(0) = 0$, $\phi'(s) > 0$ if and only if $s \in (-\epsilon, \epsilon)$ with some $\epsilon$ that will be fixed later. Note, that $\{ u_0(x) = 0\} = \{ \tilde{u}_0(x) = 0\}$ and due to the cutoff, we control the growth of the function at infinity and we have the relations $|\nabla u_0| = |\phi'(\tilde{u}_0)| |\nabla \tilde{u}_0|$ and $|\nabla u_0| \div(\frac{\nabla u_0}{|\nabla u_0|}) = \phi'(\tilde{u}_0)|\nabla \tilde{u}_0| \div(\frac{\nabla \tilde{u}_0}{|\nabla \tilde{u}_0|})$. Moreover, as $u_0$ has to lie in the stable-set, a choice for $f$ would be the following given some decreasing Lipschitz-continuous function $\eta : [0, \infty) \to [0, 1]$ that decays from one to zero,
	\[
		f(x, t) \coloneqq (-\phi'(\tilde{u}_0(x)) \div(\tfrac{\nabla \tilde{u}_0(x)}{|\nabla \tilde{u}_0(x)|}) + \phi(\tilde{u}_0(x))) \eta(t).
	\]
	The existence of a solution now follows by Theorem \ref{thm:mcf_sol}.
	
	Consider, now, for instance the submanifold $\mathbb{S}^{n-1}_r \coloneqq \{ x\in \Rn \;|\; |x| = r\}$. Given an appropriate level-set function at $|x| = r$ we have then the quantity $|\nabla u|\div(\tfrac{\nabla u}{|\nabla u|}) = 2(n-1)$ and $2|\nabla u|\partial R(u_t) = 2r \partial R(u_t)$. Hence, if $r > n-1$, then the curvature of the function is small enough so that it lies in the stable-set and the solution is stationary. On the other hand, if the curvature is big, i.e., $r < n-1$, then the submanifold has to degenerate to a point in order to satisfy the equation for all $t$. This can be generalized to other manifolds, i.e., only the points of the manifold with big enough curvature evolve.
\end{example}

\subsection{Stability Result}\label{ssec:stability}
In this section, we prove a stability result for discontinuous viscosity solutions. If stability holds then the solutions of a sequence of partial differential equations converge to a solution of the limit equation. One of the remarkable features of discontinuous viscosity solutions is that this is true under rather weak assumptions on the equations. In the case of discontinuous viscosity solutions, the stability result is also called the half-relaxed limit method \cite{Barles1987}.

The half-relaxed limits of a sequence of functions $u_n : \Omega \to \R$ which is uniformly locally bounded, i.e., on each compact $K \subset \Omega$, we have $\sup_{n\in \N} \sup_{x\in K} |u_n(x)| < \infty$, are
\[
	\hlimsup_{n\to\infty} u_n(x) \coloneqq \limsup_{n\to \infty, y\to x} u_n(y), \quad \hliminf_{n\to\infty} u_n(x) \coloneqq \liminf_{n\to \infty, y\to x} u_n(y).
\]

\begin{theorem}[Stability Theorem]\label{thm:stability}
	Let $n \in \N$ and $u_n \in \operatorname{USC}(\Omega \times I)$ a sequence of viscosity subsolutions of
	\[
		F_n(x, t, u_n, ({u_n})_t, \nabla u_n, D^2 u_n) \in \mathcal{S}_n(({u_n})_t) G_n(x, t, u_n, \nabla u_n) \text{ in } \Omega \times I,
	\]
	with $F_n: \Omega \times I \times \R \times \R \times \Rn \times \operatorname{Sym}(n) \to \R$ uniformly locally bounded, the functions $G_n : \Omega \times I \times \R \times \Rn \to [0, \infty)$ are uniformly converging to some $G$, and $\mathcal{S}_n : \R \to \mathcal{P}(\R)$ such that there exists $\mathcal{S}: \R \to \mathcal{P}(\R)$ such that whenever $a_n \to a$ and $\mu_n \in \mathcal{S}_n(a_n)$ then there exists a subsequence $\mu_{n_k}$ with $\mu_{n_k} \to \mu \in \mathcal{S}(a)$.
	If the functions $u_n$ are uniformly locally bounded on $\Omega \times I$, then $\overline{u} \coloneqq \hlimsup_{n\to \infty} u_{n}$ is a subsolution to
	\[
		\underline{F}(x, t, u, u_t, \nabla u, D^2 u) \in \mathcal{S}(u_t) G(x, t, u, \nabla u)  \text{ in } \Omega \times I,
	\]
	where $\underline{F} \coloneqq \hliminf_{n\to \infty} F_n$.
	
	The same result holds for a sequence of supersolution with the obvious modifications.
\end{theorem}

\begin{proof}
	This theorem is proved in \ref{proof:stability}.
\end{proof}

\begin{remark}
	Note that this theorem also shows that if we have a sequence of viscosity solutions in the sense of \cite{Crandall1992} to an equation of type \eqref{eq:pde} where $\mathcal{S}_n$ are not set-valued. Then they converge in the above sense to a function that is a viscosity solution in the sense of definition \ref{def:dis_vc} when the limit $\mathcal{S}$ is set-valued.
		
	This property was used in \cite{CourteDondlStefanelli} to prove regularity for viscosity solutions to a specific partial differential inclusion.
\end{remark}

\section{Relation between the Perron Solutions, the Vanishing Viscosity Solution and the Maximal Minimizing Movement Solution}\label{sec:per_sol}
In the final section of this article, we want to discuss the relation between the Perron solutions, the vanishing viscosity solution, and the maximal minimizing movement solution. The two Perron solutions $U$ and $V$ of an equation are constructed by maximizing over all subsolutions respectively minimizing over all supersolutions, see Theorem \ref{thm:perron}. Hence, they are in a sense the extremal solutions. Indeed, if we take any other discontinuous viscosity solution that lies in the given bounds, it will by definition lie between $V$ and $U$. Therefore, both solutions are of particular interest and we analyze them for the following ordinary differential inclusion
\begin{equation}\label{eq:ex2}
\begin{array}{ll}
\partial R(u_t) + e(u) \ni f & \text{ in } (0, T), \\
u(0) = 0 &,
\end{array}
\end{equation}
with $e \in C_{\rm loc}^{0, 1}(\R)$ is a locally Lipschitz-continuous function with $e(r) = 0$ if and only if $r=0$, $\lim_{r\to \pm \infty} e(r) = \pm \infty$ (e.g., $-e$ is the derivative of a tilted double-well potential), and $f : (0, T) \to \R$ is also Lipschitz-continuous and increasing with $|f(0)| \le 1$. The subdifferential $\partial R$ is given by \eqref{eq:R}.

In the following, we need to consider the monotonically increasing envelopes of $e$. They are defined as follows
\begin{align}
	\label{eq:mon_env_lower}e_m(x) \coloneqq \sup\{ m(x) \;|\; m \text{ is increasing and } m \le e \}, \\
	\label{eq:mon_env_upper}e^m(x) \coloneqq \inf\{ m(x) \;|\; m \text{ is increasing and } m \ge e \}.
\end{align}
As $e$ is continuous and hence locally bounded, they are well-defined and also increasing.

The existence of the Perron solutions $U$ and $V$, see Theorem \ref{thm:perron}, follows by constructing suitable sub- and supersolutions. To be as general as possible we have to find sub- and supersolutions that are rapidly increasing/decreasing in order to avoid that the choice of the sub-/supersolution constricts the Perron solutions. For this, we note that any differentiable, monotonically decreasing function $u$ with $u(0) = 0$ is a viscosity subsolution, as $f$ is increasing, i.e.,
$$ \partial R(u_t) + e(u) \ni -1 + e(u) \le -1 + 0 \le f(0) \le f(t).$$
Moreover, if we take any strictly monotone function $m \ge e^m$, the equation
\begin{equation}
	\begin{array}{ll}
	\partial R(v_t) + m(v) \ni f &, \text{ in } (0, T), \\
	v(0) = 0 &,
	\end{array}
\end{equation}
has a unique viscosity solution $v$ as comparison, Theorem \ref{thm:comp}, holds. Any such solution $v$ is a viscosity supersolution to \eqref{eq:ex2}, i.e., for all $t\in I$ and all $b\in \mathcal{P}^{1,-} v(t)$ there is $\nu \in \partial R(b)$ such that
$$\nu + e(v(t)) \le \nu + m(v) \le f(t).$$

If we choose a rapidly decreasing $u$ and a rapidly increasing $v$, we obtain Perron solutions $U$ and $V$ from Theorem \ref{thm:perron} and it holds $u \le V \le U \le v$.

\begin{remark}
	Any discontinuous viscosity solution to \eqref{eq:ex2} lies between $V$ and $U$ if it is bounded by $u$ and $v$. By our choice of $u$ and $v$ these bounds will always be satisfied.
\end{remark}

\subsection{Viscous Approximation}
We prove now that the smaller Perron solution, i.e., $V$, which is the infimum over all supersolutions, coincides with the discontinuous viscosity solution $u_{\rm vis}$ of \eqref{eq:ex2} which is obtained by considering the vanishing viscosity limit. Therefore, consider for $\epsilon > 0$ the following system
\begin{equation}\label{eq:ex2_vis}
	\begin{array}{ll}
	\epsilon u_t + \partial R(u_t) + e(u) \ni f &, \text{ in } (0, T), \\
	u(0) = 0 &.
	\end{array}
\end{equation}
\begin{theorem}
	For any $\epsilon > 0$, there exists a unique viscosity solution $u_\epsilon \in C(I)$ with $u_\epsilon(0) = 0$ to \eqref{eq:ex2_vis} and $u_\epsilon$ is positive and increasing.
\end{theorem}
\begin{proof}
	Let us first note, that all the prerequisites for Theorem \ref{thm:comp} and Remark \ref{rmk:comp}, i.e., \ref{it:C_seq}, \ref{it:C_epsDelta}, \ref{it:C_compact}, \ref{it:F_inc}, \ref{it:F_modulus}, \ref{it:G_Lipschitz},  \ref{it:G_modulus},  are satisfied.
	
	Therefore, we have a comparison principle. Existence and uniqueness is now a matter of constructing a sub- and a supersolution by Perron's method, Theorem \ref{thm:perron}. With a similar discussion as above, we can deduce the existence of a unique solution $u_\epsilon$.
	
	First, let us note that the zero function is a subsolution and hence $0 \le u_\epsilon$. Now, take $\delta > 0$ and define $w(t) \coloneqq u_\epsilon(t+\delta)$, then $w(0) \ge 0$. Moreover, for all $t \in I$ and $b\in \mathcal{P}^{1, -} w(t) = \mathcal{P}^{1, -} u(t+\delta)$ there exists $\nu \in \partial R (a)$ such that
	\[
		\epsilon b + \nu + e(w(t)) = \epsilon b + \nu + e(u(t+\delta)) \ge f(t+\delta) \ge f(t). 
	\]
	Recall that we assumed $f$ to be increasing. Hence, $w$ is a supersolution and therefore $u_\epsilon(t) \le w(t) = u_\epsilon(t+\delta)$. As this holds for any $\delta > 0$, $u_\epsilon$ is increasing.
\end{proof}

\begin{theorem}\label{thm:ex2_vis_approx}
	Let $u_{\rm vis}(t) \coloneqq \sup_{\epsilon > 0}u_\epsilon(t)$, then $u_{\rm vis}$ is a discontinuous viscosity solution to \eqref{eq:ex2} and it holds
	\[
		{u_{\rm vis}}_*(t) = u_{\rm vis}(t) = \hliminf_{\epsilon\to 0} u_\epsilon(t) \text{ and } u_{\rm vis}^*(t) = \hlimsup_{\epsilon\to 0} u_\epsilon(t).
	\]
	
	If $f$ is strictly increasing with $f'\ge \gamma > 0$, then $u_{\rm vis}$ is also a discontinuous viscosity solution to
	\begin{equation*}
	\partial R(u_t) + e^m(u) \ni f , \text{ in } (0, T).
	\end{equation*}
\end{theorem}

\begin{proof}
	We note that the monotonicity in $u_\epsilon$ implies that the parabolic sub- and superjet contain only positive values. This allows us to use the comparison principle to show that $u_{\epsilon'} \ge u_\epsilon$ whenever $\epsilon \ge \epsilon'$. Now, we can apply \cite[Lemma~2.18]{Bardi1997} combined with the stability result, Theorem \ref{thm:stability} to obtain the first assertion.

	The proof of the second statement is a little more involved. We consider the set $\{e^m - e \not=0\}$ and take any connected component $(u_0, u_1)$ of this set (due to the continuity of $e$ all connected components of this set are open). Moreover, as $e(x) = 0$ if and only if $x=0$, we have $0\not\in(u_0, u_1)$. Now, assume that there is a time $\overline{t} \in I$ with $u_{\rm vis}(\overline{t}) = \overline{u} \in (u_0, u_1)$. 
	
	There is a sequence $(t_\epsilon) \subset I$ with $u_\epsilon(t_\epsilon) = u_0$ and $u_\epsilon(t) \not= u_0$ for all $t > t_\epsilon$. If $\epsilon \ge \epsilon'$ then we have  $u_{\epsilon}(t_{\epsilon}) = u_0 = u_{\epsilon'}(t_{\epsilon'}) \ge u_{\epsilon}(t_{\epsilon'})$ and we thus obtain $t_{\epsilon} \ge t_{\epsilon'}$. We conclude that $(t_\epsilon)$ is a decreasing sequence, bounded from below by $0$ and therefore convergent to a limit $t_0 \in I$. It holds
	\[
		u_{\rm vis}(t_0) = \sup_{\epsilon > 0} u_\epsilon(t_0) \le \sup_{\epsilon > 0} u_\epsilon(t_\epsilon) = u_0 < u_{\rm vis}(\overline{t})
	\]
	and therefore, we have $t_0 < \overline{t}$.
	Now, define for $\lambda > 0$
	\[
	v_\epsilon(t) \coloneqq u_0 + \tfrac{2}{\pi} \arctan(\lambda \tfrac{(t-t_\epsilon)^2}{\epsilon})(u_1-u_0).
	\]
	This function is a subsolution to $\epsilon u_t +\partial R(u_t) + e(u) = f$ on $(t_\epsilon, T)$. Indeed, by noting that $e(v_{\epsilon}(t_\epsilon + s)) \le e(v_{\epsilon}(t_\epsilon))$, it holds for any $t = t_\epsilon + s \in (t_\epsilon, T)$ that
	\begin{align*}
	\frac{4\lambda}{\pi} \frac{s}{1+(\lambda \frac{s^2}{\epsilon})^2}(u_1 - u_0) + 1 + e(v_\epsilon(t_\epsilon + s)) &\le \frac{4\lambda t}{\pi }(u_1-u_0) + e(v_\epsilon(t_\epsilon)) + 1\\
	&\le \frac{4\lambda s}{\pi }(u_1-u_0) + f(t_\epsilon)
	\end{align*}
	By choosing $\lambda =  4^{-1} \gamma\pi(u_1-u_0)^{-1}$ we obtain that the right-hand side is smaller then $f(t)$. Hence by comparison, we have $v_\epsilon(t_\epsilon + \delta) \le u_\epsilon(t_\epsilon + \delta)$ for any $\delta > 0$. Fix any $\delta > 0$, then there is a subsequence such that $u_{\rm vis}(t_0 + 2\delta) = \lim_{\epsilon' \to 0} u_{\epsilon ' }(t_0 + 2\delta)$. Moreover, for $\epsilon' > 0$ small enough it holds $t_0 + 2\delta > t_{\epsilon'} + \delta$ and therefore
	\[
		u_{\rm vis}(t_0 + 2\delta) = \lim_{\epsilon' \to 0} u_{\epsilon'}(t_0 + 2\delta) \ge \limsup_{\epsilon' \to 0} u_{\epsilon'}(t_{\epsilon'} + \delta) \ge \limsup_{\epsilon' \to 0} v_{\epsilon'}(t_{\epsilon'} + \delta) = u_1
	\]
	Hence $u_{\rm vis}^*(t_0) = \inf_{\delta \to 0} u_{\rm vis}(t_0 + 2\delta) \ge u_1 > u_{\rm vis}(\overline{t})$. Due to the monotonicity of $u_{\rm vis}$, we reach a contradiction as $t_0 > \overline{t}$. Hence, $u_{\rm vis}(t) \not\in \{e^m - e \not=0\}$ for all $t \in I$.
	
\end{proof}

\begin{corollary}
	Assume that $f$ is strictly increasing with $f'\ge \gamma > 0$, then it holds
	\[
		u_{\rm vis} = V_*.
	\]
\end{corollary}

\begin{proof}
	First of all, we show that $u_{\rm vis}$ is a viscosity supersolution to \eqref{eq:ex2}. This follows from the first part of Theorem \ref{thm:ex2_vis_approx} and hence it holds that
	\[
		V \le u_{\rm vis}.
	\]
	For the other inequality note that $V_*$ is a viscosity supersolution to
	\begin{equation*}
		\partial R(u_t) + e^m(u) \ni f , \text{ in } (0, T),
	\end{equation*}
	with $e^m$ as defined in \eqref{eq:mon_env_upper}. To see this, take any $t\in I$ and any $b \in \mathcal{P}^{1, -} V_*(t)$, then there is a $\nu \in \partial R(b)$ such that
	\[
		\nu + e^m(V_*(t)) \ge \nu + e(V_*(t)) \ge f(t).
	\]
	Moreover, the function $w_\epsilon(t) \coloneqq u_{\rm vis}^*(t-\epsilon)$ is a strict subsolution to the above equation due to the second part of Theorem \ref{thm:ex2_vis_approx} and the strict monotonicity of $f$. This allows us to invoke the strict comparison principle, Theorem \ref{thm:comp}, and we obtain for all $t\in I$ that $w(t-\epsilon) \le V_*(t)$ and in the $\epsilon \to 0$ limit we obtain $w_*(t) \le V_*(t)$ which is equivalent to $u_{\rm vis}(t) \le V_*(t)$ and the statement follows.
	
\end{proof}

\subsection{Minimizing Movements}
We will now construct a minimizing movement solution \cite{Mielke2015} which is maximal among all minimizing movement solutions. In this way, we guarantee that the solution jumps as soon as possible. Hence, let us take a partition $0 = t_0 < t_1 < \dots < t_{N-1} < t_N = T$ and define
\begin{align*}
	q^0 &\coloneqq 0,  \\
	q^k &\coloneqq \sup\{q \;|\; q \in \arg\min_{\tilde{q}\in \R} \{ E(\tilde{q}) - f(t_k)\tilde{q} + |\tilde{q}-q^{k-1}| \} \},
\end{align*}
with $f$ as above, i.e., increasing and $f(0) \in [0, 1]$, and $E(a) \coloneqq \int_{0}^{a} e(s) \;\mathrm{d}s$, note that $E \ge 0$.
\begin{lemma}
	As $f$ is increasing, we have $q^k \ge q^{k-1}$ and $q^k = 0$ while $f(t_k) \le 1$.
	
	Moreover, if $f$ is strictly increasing then for all $k \ge k_0$, where $k_0$ is the first $k$ such that $f(t_k) > 1$, we have $q^k > q^{k-1}$.
\end{lemma}
\begin{proof}
	First of all, we have that
	\[
		E(q^1) - f(t_1)q^1 + |q^1| \le 0
	\]
	and as $E$ is positive $q^1$ has to be positive. Moreover, it holds for $k \ge 1$ that
	\begin{align*}
		E(q^k) - f(t_k)q^k + |q^k - q^{k-1}| 
		&\le E(q^{k-1}) - f(t_k)q^{k-1} \\
		&\le E(q^{k}) - f(t_{k-1})q^{k} + (f(t_{k-1}) - f(t_k))q^{k-1} \\&~~+ |q^k - q^{k-2}| - |q^{k-1} - q^{k-2}|.
	\end{align*}
	Hence, we obtain
	\[
		(f(t_{k-1}) - f(t_k))q^k \le (f(t_{k-1}) - f(t_k))q^{k-1}
	\]
	and as $f(t_{k-1}) - f(t_k) \le 0$, we conclude that $q^k \ge q^{k-1}$.

	If $f(t_1) \le 1$, then for any $\tilde{q} > 0$ we have $E(\tilde{q})-f(t_1)\tilde{q} + |\tilde{q}| \ge E(\tilde{q}) - \tilde{q} +\tilde{q} = E(\tilde{q}) > 0$, as $e$ is increasing around $0$. This shows that $0$ is the only minimizer. Now, we can iterate this argument to see that $q^k = 0$ while $f(t_k) \le 1$.
	
	The final statement can again be proven using induction. First, we show that $q^{k_0} > q^{k_0-1} = 0$. As $f(t_0) = 1+\epsilon$ for some $\epsilon > 0$, we obtain for all $\delta > 0$ that
	\[
		E(\delta) - f(t_0)\delta + \delta = \int_{0}^{\delta} e(s) \;\mathrm{d}s - \epsilon \delta \le \tfrac{L}{2}\delta^2 -\epsilon \delta,
	\]
	where $L=L(0) > 0$ is the Lipschitz-constant of $e$ around $0$. Now, we can choose $\delta < \frac{2\epsilon}{L}$ and obtain
	\[
		E(\delta) - f(t_0)\delta + \delta < 0
	\]
	and hence $0$ cannot be a minimizer which implies $q^{k_0} > 0$. Now, take any $k \ge k_0$ and assume that $q^{k-1} > q^{k-2} > \dots > q^{k_0} > 0$.
	Again, we have for any $\delta > 0$ that
	\begin{align*}
		E(q^{k}) - f(t_{k})q^{k} + (q^{k}-q^{k-1}) &\le E(q^{k-1}+\delta) - f(t_{k})(q^{k-1}+\delta) + \delta \\
		&= E(q^{k-1}) + \int_{q^{k-1}}^{q^{k-1}+\delta} e(s) \;\mathrm{d}s - f(t_{k})(q^{k-1}+\delta) + \delta, \\
		&\le E(q^{k-1}) + \delta e(q^{k-1}) + \tfrac{L}{2}\delta^2 - f(t_{k})(q^{k-1}+\delta) + \delta,
	\end{align*}
	where $L=L(q^{k-1})$ is the Lipschitz-constant of $e$ in some interval around $q^{k-1}$, say $[q^{k-1}-1, q^{k-1}+1]$. We reorder the inequality and use $E(q^k) - E(q^{k-1}) \ge 0$ to obtain 
	\[
		(1-f(t_{k})) (q^k-q^{k-1}) \le \delta (e(q^{k-1})- f(t_{k})q^{k-1} + 1 +  \tfrac{L}{2}\delta).
	\]
	Again, we rewrite $f(t_k) = f(t_{k-1}) + \epsilon$ and we use that $e(q^{k-1}) - f(t_{k-1})q^{k-1} + 1 = 0$ as $q^{k-1}$ is a minimizer and $q^{k-1} > q^{k-2}$ to conclude
	\[
		(1-f(t_{k})) (q^k-q^{k-1}) \le \delta (- \epsilon q^{k-1} +  \tfrac{L}{2}\delta).
	\]
	Finally, we can choose $\delta$ so small enough so that the right hand-side becomes negative and then divide by $(1-f(t_k))<0$ in order to show that
	\[
		q^k-q^{k-1} > 0.
	\]
\end{proof}

As viscosity solutions are not defined on discrete time-points, we have to consider some interpolation of the discrete values. Let us therefore define the interpolant
\begin{align*}
	q^N(t) \coloneqq q^{k-1} \text{ for } t \in [t_{k-1}, t_k).
\end{align*}

\begin{lemma}\label{lem:mm_inc}
	The interpolant $q^N$ is a discontinuous viscosity solution of
	\begin{align*}
		\partial R(u_t) + e(u) \ni f^N,
	\end{align*}
	where $f^N(t) \coloneqq f(t_{k-1})$ for $t \in [t_{k-1}, t_k)$.
	Moreover, if $f$ is strictly increasing then $q^N$ is also a discontinuous viscosity solution of 
	\begin{align*}
	\partial R(u_t) + e_m(u) \ni f^N.
	\end{align*}
\end{lemma}
\begin{proof}
	As $q^k \in \arg\min_{\tilde{q}\in \R} \{ E(\tilde{q}) - f(t_k)\tilde{q} + |\tilde{q}-q^{k-1}| \}$, it follows by the properties of subdifferentials that
	\begin{align}
	\label{eq:proof_mm_incl} 0 &\in [-1 + e(q^k) - f(t_k), 1+ e(q^k) - f(t_k)]&\text{ if } q^k = q^{k-1}, \\
	\label{eq:proof_mm_eq} 0 &= 1 + e(q^k) - f(t_k) & \text{ if } q^k > q^{k-1},
	\end{align}
	for all $k \in 1, \dots, N$.
	
	Whenever, $t\in I$ is a continuity point of $q^N$ then $\mathcal{P}^{1, +}q^N(t) = \mathcal{P}^{1, -}q^N(t)= \{0\}$ and due to the inclusion \eqref{eq:proof_mm_incl}, $q^N$ satisfies the sub- and supersolution inequalities. Let us now look at the jump points, i.e., $q^{k} = (q^N)^*(t) > (q^N)_*(t) = q^{k-1}$. In this case it holds $\mathcal{P}^{1, +}(q^N)^*(t) = \mathcal{P}^{1, -}(q^N)_*(t)= [0, \infty)$. Hence, we obtain for all $a \in \mathcal{P}^{1, +}(q^N)^*(t)$ with \eqref{eq:proof_mm_eq} that
	\begin{align*}
		\partial R(a) + e((q^N)^*(t)) &\ni 1 + e(q^{k}) = f(t_k) = f^N(t) \le (f^N)^*(t),
	\end{align*}
	i.e., $(q^N)^*$ is a viscosity subsolution. On the other hand, for all $b \in \mathcal{P}^{1, -}(q^N)_*(t)$ we have
	\begin{align*}
	\partial R(b) + e((q^N)_*(t)) &\ni 1 + e(q^{k-1}) \ge f(t_{k-1}) = (f^N)_*(t),
	\end{align*}
	and therefore $(q^N)_*$ is a viscosity supersolution and the function $q^N$ is a discontinuous viscosity solution.
	
	Hence, the first statement holds. As $q^k$ is chosen to be the maximal value of the minimizers, the function $e$ in the equations above can be replaced by $e_m$. Indeed, as long as $t_k < t_{k_0}$, we have $q^k = 0$ and by assumption we also have that $e(0) = e_m(0)$. For any, $t_k \ge t_{k_0}$ it holds due to Lemma \ref{lem:mm_inc} that
	\[
		e(q^k)-f(t_k) + 1 = 0,
	\]
	and by the definition of $q^k$ this is equivalent to $q^k = \max\{e^{-1}(-1+f(t_k)) \}$ (note that we used the coercivity of $e$ here). However, this and the local Lipschitz-continuity of $e$ imply that there is a monotone function that lies below $e$ and touches $e$ at $q^k$ and hence by the definition of $e_m$ it follows that $e(q^k) = e_m(q^k)$,
\end{proof}

Now, be refining the partition we can analyze the limit $N \to \infty$. Due to the stability result \ref{thm:stability}, the function $u_{\rm mm} \coloneqq \hliminf_{N\to\infty}q^N$ is a viscosity supersolution of
\begin{align*}
	\partial R(u_t) + e_m(u) \ni f.
\end{align*}

\begin{theorem}
	If $f$ is strictly increasing, then $u_{\rm mm} = U_*$, where $U$ is the greater of the Perron solutions of \eqref{eq:ex2}.
\end{theorem}
\begin{proof}
	First of all, we notice that $q^N$ is a viscosity subsolution of \eqref{eq:ex2} as $f$ is increasing, i.e.,
	\[
		\partial R(q^N_t) + e(q) \le f^N \le f.
	\]
	Therefore $q^N \le U$ and hence $u_{\rm mm} \le U_*$. One the other hand, if we set $h(t) \coloneqq u_{\rm mm}(t+\epsilon)$ for some $\epsilon > 0$ then
	\[
		\partial R(h_t) + e_m(h) \ni f(t+\epsilon) > f(t),
	\]
	i.e., $h$ is strict supersolution to this equation and $U$ is a subsolution as
	\[
		\partial R(U_t) + e_m(U) \le \partial R(U_t) + e(U) \le f(t).
	\]
	The comparison principle \ref{thm:comp} allows us to conclude that $U(t) \le u_{\rm mm}(t+\epsilon)$. The asserted statement follows by taking the limit as $\epsilon \to 0$.
\end{proof}

\subsection{Hysteresis}\label{ssec:hysteresis}
\begin{figure}
	\begin{subfigure}{0.45\textwidth}
		\includegraphics[width=\textwidth]{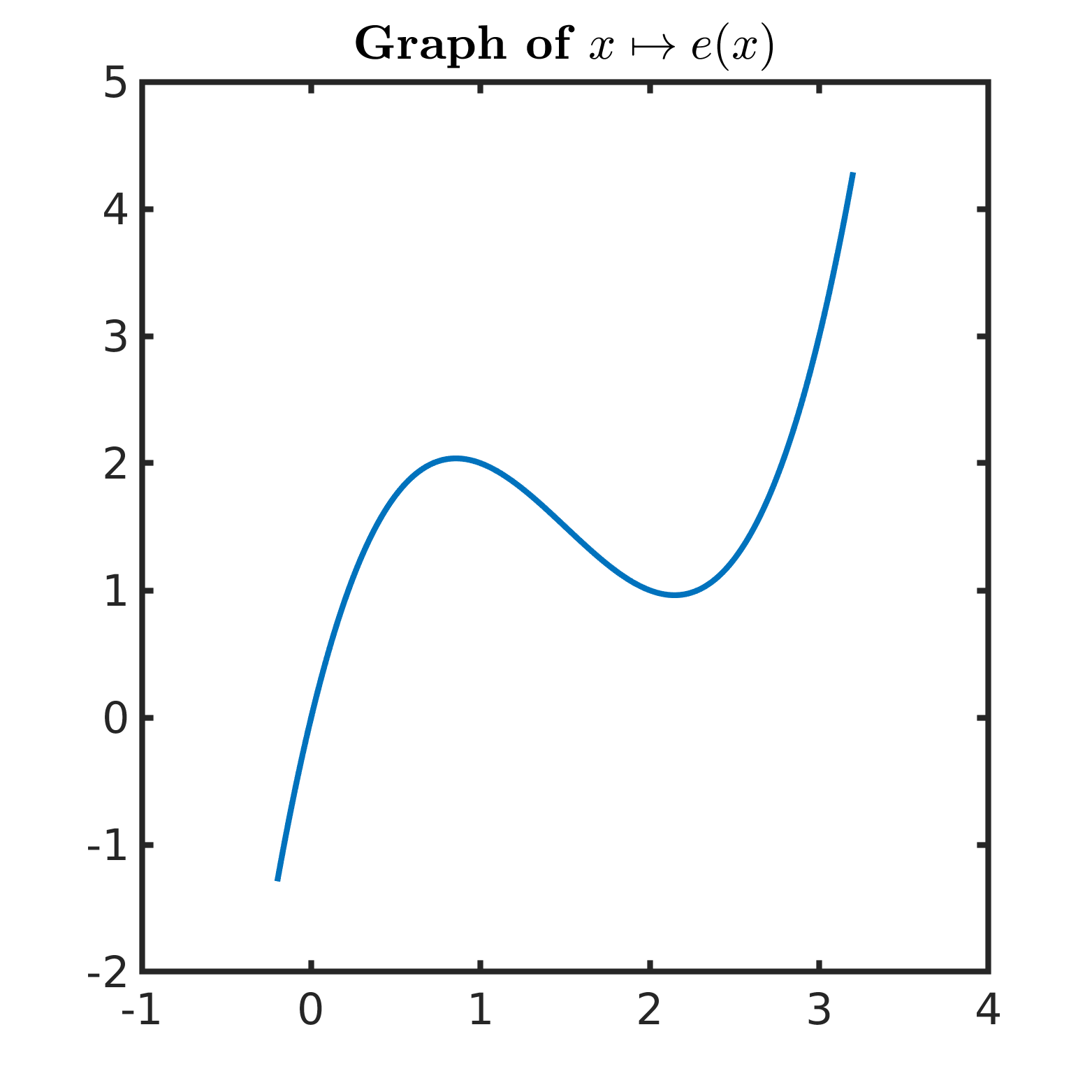}
	\end{subfigure}
	\begin{subfigure}{0.45\textwidth}
		\includegraphics[width=\textwidth]{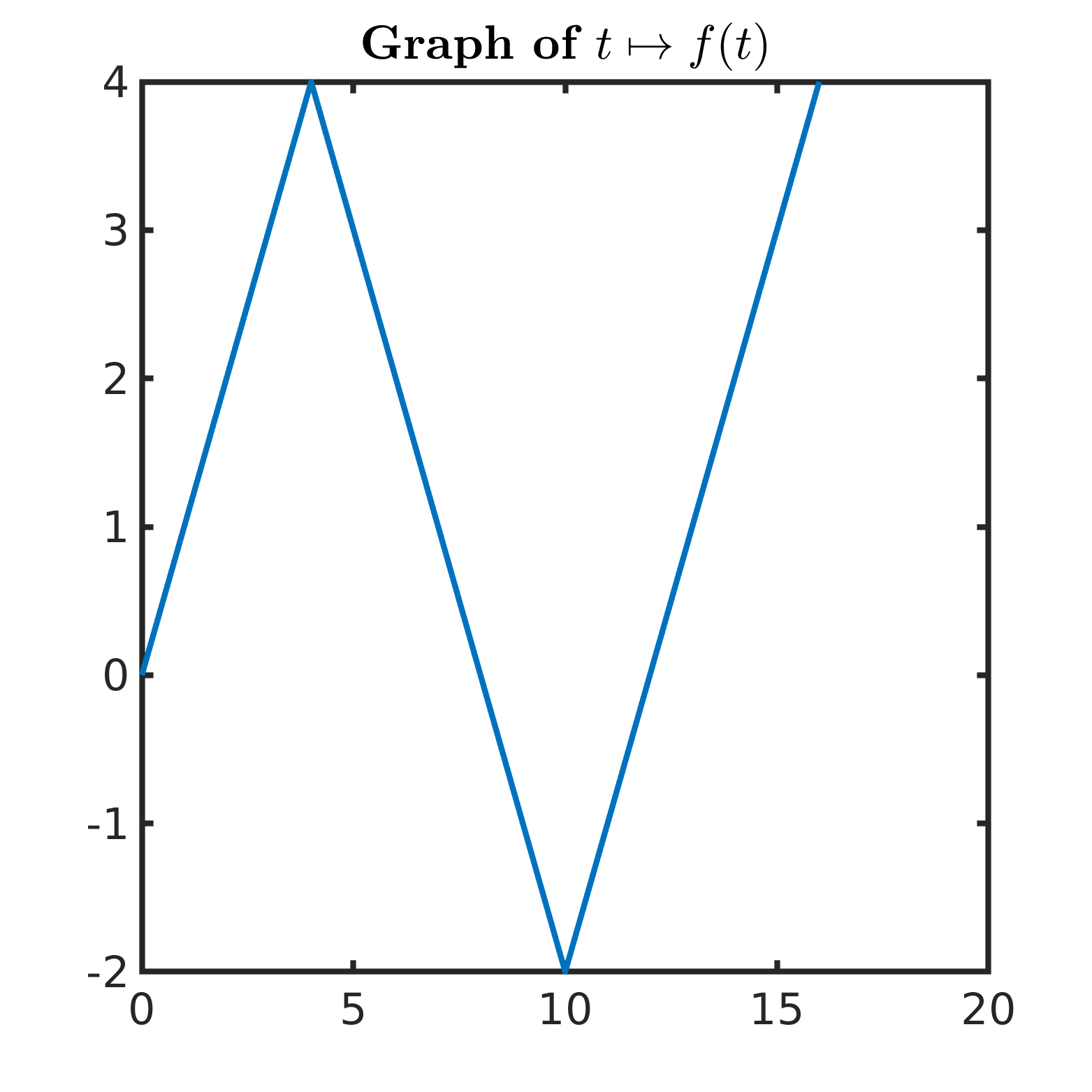}
	\end{subfigure}
	\begin{subfigure}{0.45\textwidth}
		\includegraphics[width=\textwidth]{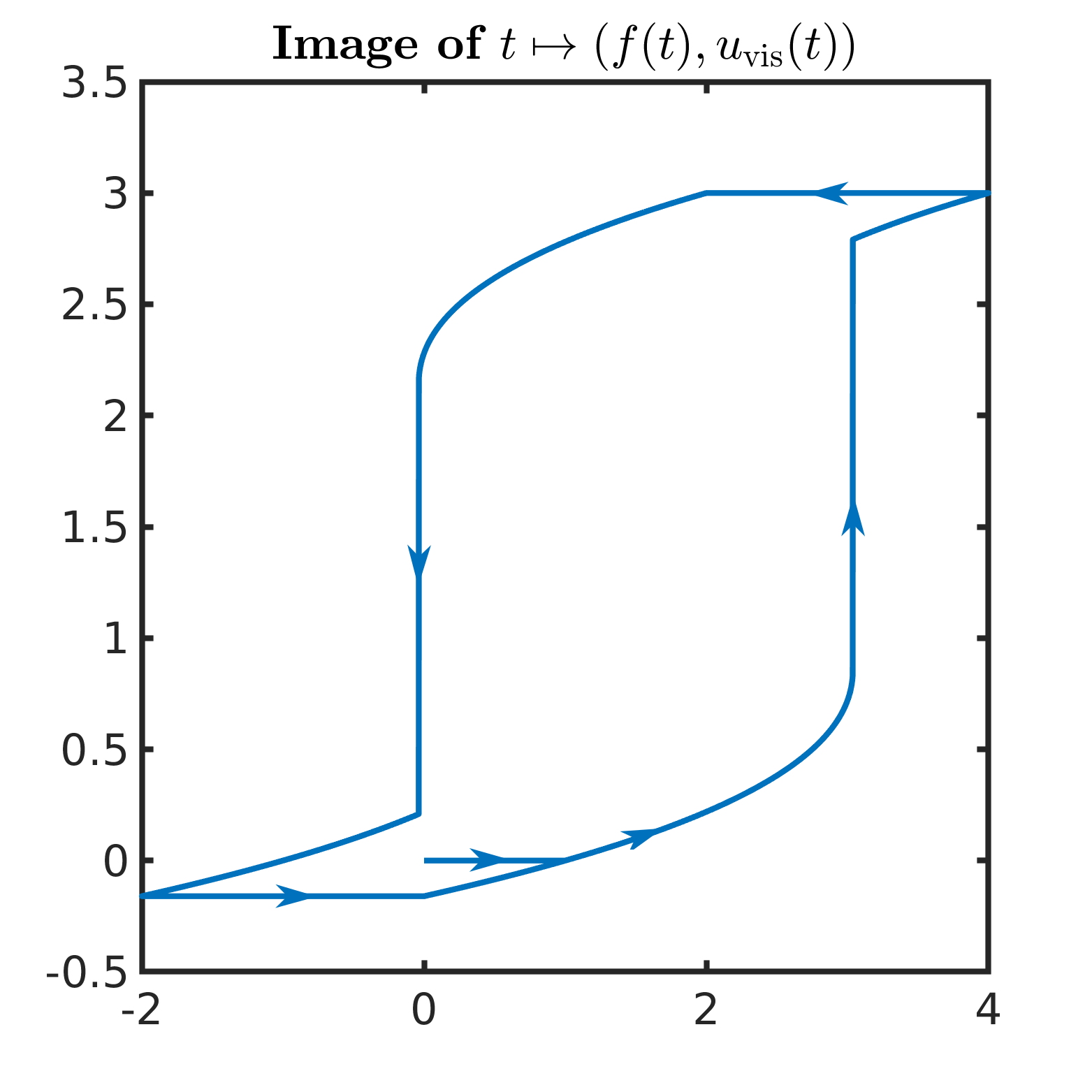}
	\end{subfigure}
	\begin{subfigure}{0.45\textwidth}
		\includegraphics[width=\textwidth]{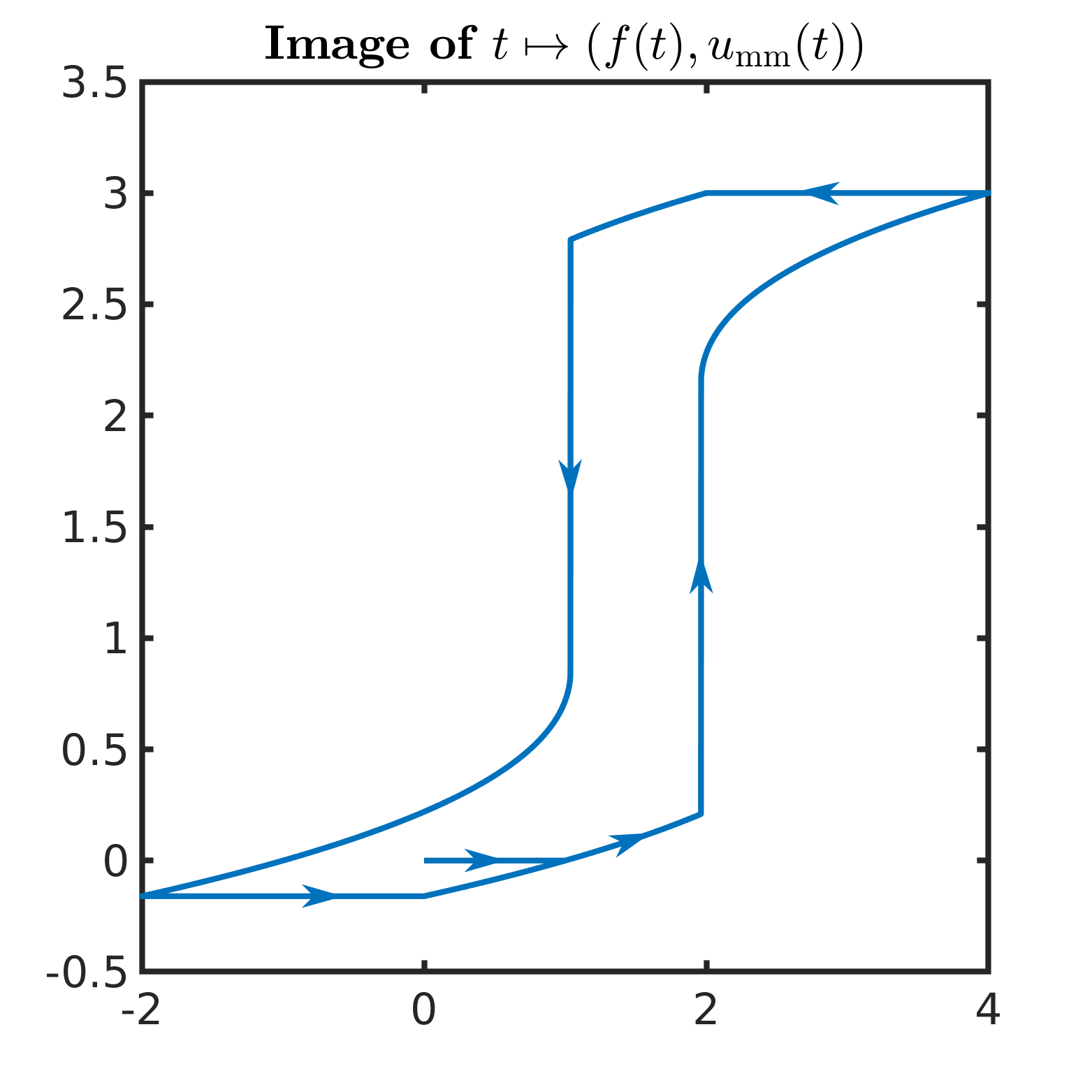}
	\end{subfigure}
	\caption{The top figures show the graph of the functions $e$ and $f$ from subsection \ref{ssec:hysteresis}. The curves $t \mapsto (f(t), u_{\rm vis}(t))$ and $t \mapsto (f(t), u_{\rm mm}(t))$ in the lower two figures show the emergence of a rate-independent hysteresis loop which is due to the non-monotone $e$. The arrows in the bottom right and bottom left figure show the evolution in time.}\label{fig:hysteresis}
\end{figure}

We have characterized the Perron solutions under an increasing loading $f$. If we assume that $f$ is decreasing then the previously established correspondence changes, i.e., $U$ corresponds to the vanishing viscosity solution and $V$ corresponds to the maximal minimizing movement solution. 

To see the emergence of a rate-independent hysteresis loop, we assume that $f$ is a periodic loading that changes direction outside of the set 
$$[ \min\{e^m-e_m \not= 0\} - 1, \max\{e^m-e_m \not= 0\} + 1].$$
In this case, when $f$ changes direction, we have $U=V$ and the correspondence changes as described above.

However, we have also seen that $V < U$ in the set $\{e^m-e_m \not= 0\}$ and hence we obtain a (or multiple) rate-independent hysteresis loop for $u_{\rm vis}$ and $u_{\rm mm}$. Note that there are to different effects that contribute to the hysteresis loop, first due to the term $\partial R$ we obtain hysteresis when we change loading. Second, due to the non-monotone $e$, we obtain a further hysteresis loop. To highlight this, we refer to Figure \ref{fig:hysteresis} where the images of the curves $t \mapsto (f(t), u_{\rm vis}(t))$ and $(f(t), u_{\rm mm}(t))$ are plotted for $e : \R \to \R$ given by $e(x) \coloneqq x^3 - \frac{9}{2} x^2 + \frac{11}{2} x$ and  $f : [0, 16] \to \R$ given by
\[
	f(t) = \left\{\begin{array}{ll}
	t &\text{if } t \le 4, \\
	8-t &\text{if } t \in (4, 10), \\
	t-12 &\text{if } t \ge 10.
	\end{array}\right.
\]
Note that $e$ achieves has a local maximum at $x_{\rm max} \coloneqq \frac{3}{2} - \frac{1}{2}\sqrt{\frac{5}{3}}$ with $$e(x_{\rm max}) = \frac{1}{36} (54+5\sqrt{15}) \approx 2.038$$ and a local minimum at $x_{\rm min} \coloneqq \frac{3}{2} + \frac{1}{2}\sqrt{\frac{5}{3}}$ with $$e(x_{\rm min}) = \frac{1}{36} (54-5\sqrt{15}) \approx 0.962$$
and hence, we expect $u_{\rm vis}$ to jump first when $f(t) = e(x_{\rm max}) + 1 \approx 3.038$ and second when $f(t) = e(x_{\rm min}) - 1 \approx -0.038$. Furthermore, we expect $u_{\rm mm}$ to have a first jump when $f(t) = e(x_{\rm min}) + 1 \approx 1.962$ and a second jump when $f(t) = e(x_{\rm max}) - 1 \approx 1.038$. 

\section*{Acknowledgement}
LC acknowledges support from the Luxembourg National Research Fund (FNR) (13502370).

\section*{Declaration of Interest}
Declarations of interest: none.

\appendix
\section{Proofs}\label{app:proof}

\subsection{Perron's method}
The next two lemmas turn out to be helpful while proving theorem \ref{thm:perron}.
\begin{lemma}\label{lem:sup_subsolutions}
	Let $\mathcal{F}$ be a non-empty set of subsolutions of \eqref{eq:pde} and set
	\[
	w(x, t) \coloneqq \sup_{u\in \mathcal{F}} u(x, t).
	\]
	If $w(x, t) < \infty$ then its upper-semicontinuous envelope $w^*$ is a subsolution of \eqref{eq:pde}. The analogous result holds for supersolutions.
\end{lemma}
\begin{proof}
	It is enough to prove the result for subsolution as each supersolution is a subsolution of $-F(x, t, u, u_t, \nabla u, D^2 u) \in -\mathcal{S}(a) G(x, t, u, \nabla u)$, i.e., by replacing $F$ with $-F$ (noting that the degenerate ellipticity assumption is not used in the proof) and $\mathcal{S}$ with $-\mathcal{S}$ we obtain the result for supersolutions.

	Let $(x, t) \in \Omega \times I$ and $(a, p, X) \in \mathcal{P}^{2, +}w^*(x, t)$. By definition of the upper-semicontinuous envelope one can find a sequence $(x_n, t_n, u_n) \in \Omega \times I \times \mathcal{F}$ such that $u_n(x_n, t_n) \to w^*(x, t)$ and for any other sequence $(x'_n, t'_n) \to (x, t)$ we have $\limsup_{n\to \infty} u_n(x'_n, t'_n) \le w^*(x, t)$. This implies that there is a sequence $(\hat x_n, \hat t_n) \in \Omega \times I$ and a sequence $(a_n, p_n, X_n) \in \mathcal{P}^{2, +} u_n(\hat x_n,\hat t_n)$ such that
	\[
	(\hat x_n, \hat t_n, u_n(\hat x_n, \hat t_n), a_n, p_n, X_n) \to (x, t, w^*(x, t), a, p, X),
	\]
	see \cite[Proposition~4.3]{Crandall1992}. As $u_n \in \mathcal{F}$ is a subsolution there is a $\mu_n \in \mathcal{S}(a_n)$ such that
	\[
	F_*(\hat x_n, \hat t_n, u(\hat x_n, \hat t_n), a_n, p_n, X_n) \le \mu_n G(\hat x_n, \hat t_n, u(\hat x_n, \hat t_n), p_n).
	\]
	Passing to a subsequence, by \ref{it:C_seq} we have $\mu_{n_k} \to \mu \in \mathcal{S}(a)$ and the last inequality still holds for this subsequence. Due to the lower semicontinuity of $F_*$ and the continuity of $G$, we can pass to limit inferior as $k\to \infty$ and we obtain
	\[
	F_*(x,t, w^*(x, t), a, p, X) \le \mu G(x, t, w^*(x, t), p).
	\]
	We proved that for all $(x, t) \in \Omega \times I$ and all $(a, p, X) \in \mathcal{P}^{2, +} w^*(x, t)$ there is a $\mu \in \mathcal{S}(a)$ such that the last inequality holds. This shows that $w^*$ is a subsolution of \eqref{eq:pde}.
\end{proof}

\begin{lemma}\label{lem:bump_const}
	Let $u$ be a subsolution of \eqref{eq:pde}. Assume that $u_*$ is not a supersolution at some point $(\hat x, \hat t)$, i.e., there exists $(a, p, X) \in \mathcal{P}^{2, -} u_*(\hat x, \hat t)$ such that for all $\mu \in \mathcal{S}(a)$ we have
	\begin{equation}
	F^*(\hat x, \hat t, u_*(\hat x, \hat t), a, p, X) < \mu G(\hat x, \hat t, u_*(\hat x, \hat t), p). \label{eq:strict_ineq}
	\end{equation}
	In this case, for any $\epsilon > 0$ there exists a subsolution $u_\epsilon: \Omega \times I \to \R$ satisfying
	\begin{itemize}
		\item $u_\epsilon(x, t) \ge u(x, t)$,
		\item $\sup(u_\epsilon - u) > 0$,
		\item $u_\epsilon(x, t) = u(x, t)$ for all $(x, t)\in \Omega\times I$ with $|(x, t)-(\hat{x}, \hat{t})|\ge\epsilon$.
	\end{itemize}
	
	Again, the same result holds also for supersolutions with the obvious modifications.
\end{lemma}
\begin{proof}
	Let $(\hat{x}, \hat{t})$ and $(a, p, X) \in \mathcal{P}^{2, -} u_*(\hat{x}, \hat{t})$ be such that inequality \eqref{eq:strict_ineq} holds. By \ref{it:C_compact} the set $\mathcal{S}(a)$ is compact and there is $\alpha > 0$ such that
	\[
	F^*(\hat x, \hat t, u_*(\hat x, \hat t), a, p, X) - \mu G(\hat x, \hat t, u_*(\hat x, \hat t), p) \le -\alpha
	\]
	for all $\mu \in \mathcal{S}(a)$. Let us define
	\begin{equation}\label{eq:bump_const_proof_1}
	u_{\delta, \gamma}(x, t) \coloneqq u_*(\hat x, \hat t) + \delta + a(t-\hat{t})+ \left<p, x-\hat x\right>+ \tfrac{1}{2}\left<X(x-\hat x), x-\hat x\right> - \tfrac{\gamma}{2} |(x, t) - (\hat{x}, \hat{t})|^2.
	\end{equation}
	As $F^*$ is upper-semicontinuous, we have for $|(x, t) - (\hat{x}, \hat{t})| \to 0$,
	\begin{align*}
	&F^*( x,  t, u_{\delta, \gamma}( x, t), (u_{\delta, \gamma})_t( x, t), \nabla u_{\delta, \gamma}(x, t), D^2 u_{\delta,\gamma}(x, t)) &\\&= F^*( x,  t, u_{\delta, \gamma}( x, t), a - \gamma (t-\hat t), p - \gamma (x-\hat{x}), X - \gamma \operatorname{Id}) \\
	&\le F^*(\hat x, \hat t, u_*(\hat x, \hat t), a, p, X) + {\rm{o}}(1).
	\end{align*}
	The continuity of $G$ can be used in the same way to obtain 
	$$
	G( x,  t, u_{\delta, \gamma}( x, t), \nabla u_{\delta, \gamma}(x, t))= G(\hat x, \hat t, u_*(\hat x, \hat t), p) + {\rm{o}}(1).
	$$
	Moreover, $\partial_t u_{\delta, \gamma}( x, t) = a -\gamma (t-\hat{t})$ which implies with condition \ref{it:C_epsDelta} that there are $\mu_\gamma \in \mathcal{S}(\partial_t u_{\delta, \gamma}( x, t))$ with $\gamma$ small enough such that $\mu_\gamma = \mu + o(1)$. As \eqref{eq:bump_const_proof_1} holds, we conclude that if $\delta, \gamma, r$ are small enough then is $u_{\delta, \gamma}$ a subsolution of \eqref{eq:pde} in $B_r(\hat{x}, \hat{t})$. Moreover, since
	\[
	u(x, t) \ge u_*(x, t) \ge u_*(\hat{x}, \hat{t})+ a(t-\hat t)+ \left<p, x-\hat x\right>+ \tfrac{1}{2}\left<X(x-\hat x), x-\hat x\right> + {\rm{o}}(|(x, t)-(\hat x, \hat t)|^2),
	\]
	we can choose $\delta = c(\gamma, r)$ to obtain $u(x, t) > u_{\delta, \gamma}(x, t)$ for $(x, t) \in B_r(\hat{x}, \hat{t}) \setminus B_{\frac{r}{2}}(\hat{x}, \hat{t})$. Therefore the function
	\[
	u_{\gamma}(x, t) \coloneqq \left\{\begin{array}{ll}
	\max\{ u(x, t), u_{\delta, \gamma}(x, t) \} & \text{ in } B_r(\hat{x}, 	\hat{t}), \\
	u(x, t) & \text{ elsewhere.}
	\end{array}\right.
	\]
	is a subsolution by Lemma \ref{lem:sup_subsolutions}. It is clear that $u_\gamma(x, t) \ge u(x, t)$ and that in a neighborhood of $(\hat{x}, \hat{t})$ we have $u_\gamma(x, t) > u(x, t)$. For $\epsilon$ given, by choosing $r, \gamma < \epsilon$ we have that $u_{\gamma}^*$ satisfies all the required properties.
\end{proof}

\begin{proof}[Proof of Theorem \ref{thm:perron}]\label{proof:perron}
	As $U \le v \le \infty$, Lemma \ref{lem:sup_subsolutions} implies that $U^*$ is a subsolution. Now assume that $(U^*)_* = U_*$ is not a supersolution, then there has to be a neighborhood where $U_* < v$. In this neighborhood we can apply Lemma \ref{lem:bump_const} to obtain subsolutions $u_\epsilon$ that are strictly bigger then $U_*$. Moreover, as $\epsilon \to 0$ these subsolutions $u_\epsilon$ converge to $U_*$. Hence, we can choose $\epsilon$ small enough so that $U_* \le u_\epsilon < v$. This contradicts the maximality of $U$. We showed that $U_*$ is a supersolution and $U^*$ is a subsolution, i.e., $U$ is indeed a discontinuous viscosity solution.

	The same arguments also prove that $V$ is a discontinuous viscosity solution. To see the last statement in this Theorem, we note that $V^*$ is a subsolution and $U_*$ is a supersolution with $V^* \le v^* = v$ and $u = u_* \le U_*$.
\end{proof}

\subsection{Comparision Principles}
The proofs of the comparision principles, both bounded and unbouned, in this section are primarily an adaption of the methods in \cite{Crandall1992}.
\begin{proof}[Proof of Theorem \ref{thm:comp}]\label{proof:comp}
	Assume by contradiction that comparison does not hold, i.e.
	\[
	\sup_{\substack{x\in \Omega \\ t\in I}} \left\{ u(x, t) - v(x, t)\right\} \eqqcolon \delta > 0
	\]
	and define
	\[
	M_{\alpha, \gamma} \coloneqq \sup_{\substack{x, y \in \Omega \\ t\in I}} \left\{ u(x, t) - v(y, t) - \alpha|x-y|^4 - \tfrac{\gamma}{T-t}\right\}.
	\]
	We have $M_{\alpha, \gamma} > \delta/2$ for $\gamma$ small enough.
	Since the domain is bounded, the supremum is achieved at a point $(\hat{x}, \hat{y}, \hat{t}) \in \overline{\Omega}\times \overline{\Omega} \times [0,T)$.
	
	We will now show that the triplet $(\hat{x}, \hat{y}, \hat{t})$ is in the interior of the parabolic domain if $\alpha$ is large enough. Assume first that $\hat t = 0$ then
	\[
	M_{\alpha, \gamma } = u(\hat{x}, 0) - v(\hat{y}, 0) - \alpha|\hat x-\hat y|^4 - \tfrac{\gamma}{T-\hat t} \le  - \alpha|\hat x-\hat y|^4 - \tfrac{\gamma}{T-\hat t} \le 0,
	\]
	since $u\le v$ on the parabolic boundary $\partial_P (\Omega \times I)$, contradicting $M_{\alpha, \gamma} > \delta/2$. We now check that, if $\alpha$ is chosen to be large enough, $\hat{x}$ and $\hat{y}$ necessarily belong to $\Omega$. Assume the contrary, namely there exists a subsequence $\alpha_n \to \infty$ with $\hat{x}_n \in \partial \Omega$ realizing the sup. Then, we also have $\hat{y}_n \to \hat{y}_\infty \in \partial \Omega$ and therefore
	\[
	\lim_{n\to \infty}M_{\alpha_n, \gamma } = \lim_{n\to \infty}\left(u(\hat{x}_n, \hat{t}) - v(\hat{y}_n, \hat{t}) - \alpha_n|\hat x_n-\hat y_n|^4 - \tfrac{\gamma}{T-\hat t}\right) \le 0 - \tfrac{\gamma}{T} \le 0,
	\]
	where we used again that $u \le v$ on $\partial_P(\Omega\times I)$ and reached a contradiction. Therefore, we have proved that $(\hat{x}, \hat{y}, \hat{t}) \in \Omega \times \Omega \times (0, T)$, at least if $\alpha$ is large enough.  Hence, we have \cite[Theorem 8.3]{Crandall1992}
	\[
	(a, p, X) \in \mathcal{P}^{2, +} u(\hat{x}, \hat{t}) \text{ and }
	(b, p, Y) \in \mathcal{P}^{2, -} v(\hat{y}, \hat{t}) 
	\]
	with $a-b = \tfrac{\gamma}{(T-\hat{t})^2}$, $p \coloneqq 4\alpha |\hat x-\hat y|^2 (\hat x-\hat y)$, and
	\[
	-4||Z|| \left(\begin{array}{cc}\operatorname{Id} & 0 \\ 0 & \operatorname{Id}\end{array}\right) \le \left(\begin{array}{cc}X & 0 \\ 0 & -Y\end{array}\right)\le \left(\begin{array}{cc}Z+\frac{1}{2||Z||}Z^2 & -(Z+\frac{1}{2||Z||}Z^2) \\ -(Z+\frac{1}{2||Z||}Z^2 ) & Z+\frac{1}{2||Z||}Z^2 \end{array}\right),\]
	with $Z \coloneqq 4 \alpha |\hat x-\hat y|^2 \operatorname{Id} + 8\alpha (\hat x-\hat y) \otimes (\hat x-\hat y)$. This means that condition \ref{it:F_modulus} and condition \ref{it:G_modulus} can be used. As $u$ is a subsolution and $v$ is a supersolution, we can find $\mu \in \mathcal{S}(a)$ and $\nu \in \mathcal{S}(b)$ such that
	\begin{align}
	F_*(\hat x,\hat t, u, a, p, X) - \mu G(\hat x,\hat  t, u, p) \le 0,\label{eq:sub} \\
	F^*(\hat y,\hat t, v, b, p, Y) - \nu G(\hat y,\hat  t, v, p) \ge 0,\label{eq:super}
	\end{align}
	where one of the inequalities holds even if we replace $0$ by $\mp\lambda$ with $\lambda > 0$.
	By subtracting \eqref{eq:sub} from \eqref{eq:super}, we obtain
	\begin{align*}
	\lambda \le&~ F^*(\hat y,\hat t, v, b, p, Y) - F_*(\hat x,\hat t, u, a, p, X)- \nu G(\hat y,\hat  t, v, p) + \mu G(\hat x,\hat  t, u, p).
	\end{align*}
	Adding and subtracting terms, we get
	\begin{align*}
	\lambda \le&~ F^*(\hat y,\hat t, v, b, p, Y) - F_*(\hat x,\hat t, v, a, p, X) \\
	&+ F_*(\hat x,\hat t, v, a, p, X)- F_*(\hat x,\hat t, u, a, p, X) \\
	&- \nu G(\hat y,\hat  t, v, p)+ \nu G(\hat x,\hat  y, v, p) \\
	&- \nu G(\hat x,\hat  t, v, p)+ \nu G(\hat x,\hat  t, u, p) 
	\\&- \nu G(\hat x,\hat  t, u, p) + \mu G(\hat x,\hat  t, u, p) \\
	\le&~ \omega_F(|\hat x-\hat y| + \alpha |\hat x-\hat y|^4) + \mathcal{S}_{\mathrm{max}}L_G(v-u) \\
	&+|\nu| \omega_G(|\hat x-\hat y|+ \alpha |\hat x-\hat y|^4) 
	+ \mathcal{S}_{\mathrm{max}} L_G |u-v| \\
	&+ (\mu - \nu) G(\hat x,\hat  t, u, p).
	\end{align*}
	Where the second inequality follows from \ref{it:F_inc}, \ref{it:F_modulus},   and \ref{it:G_Lipschitz}, \ref{it:G_modulus} by noticing that, $X \le Y$, and $a > b$. In particular, \ref{it:S_mon} implies that $\mu-\nu \le 0$. As $G \ge 0$, it follows that the last term above is negative. Eventually, by means of $u(\hat x, \hat t) > v(\hat y, \hat t)$ we get
	\begin{align*}
	\lambda \le&~\omega_F(|\hat x-\hat y|+ \alpha |\hat x-\hat y|^4)+\mathcal{S}_{\mathrm{max}}\omega_G(|\hat x-\hat y|+ \alpha |\hat x-\hat y|^4).
	\end{align*}
	By taking $\alpha \to \infty$, we have $\alpha |\hat x-\hat y|^4 \to 0$ (see \cite{Crandall1992}) which also implies $|\hat x-\hat y| \to 0$. Therefore, the right-hand side above goes to $0$, leading to a contradiction.
\end{proof}

\begin{proof}[Proof of Theorem \ref{thm:comp_rn}]\label{proof:comp_rn}
	We are going to subdivide the proof in two steps. First, we are going to prove that the difference $u-v$ satisfies a growth estimate and afterwards we will prove the comparison principle.
	
	\emph{Step 1: Growth estimate}
	
	The first step of the proof consists in proving that the difference $u-v$ satisfies the growth estimate 
	\begin{equation}
	\sup_{(x, y, t) \in \Rn \times \Rn \times I} u(x, t) - v(y, t) - 2\eta^{-1}K|x-y| - \tfrac{\gamma}{T-t} < \infty, \label{eq:comp_rn_growth}
	\end{equation}
	where $K \coloneqq K_{F} + \mathcal{S}_{\textrm{max}} K_G$. Following \cite[Theorem~5.1.]{Crandall1992}, we choose a family $\beta_R$ of $C^2(\Rn)$ functions such that
	\begin{enumerate}[label=$\roman*)$]
		\item $\beta_R \ge 0$,
		\item $\liminf_{|x| \to \infty} \frac{\beta_R(x)}{|x|} \ge 2L$,
		\item $|D\beta_R(x)| + |D^2\beta_R(x)| \le C$, for $R \ge 1$, $x\in \Rn$,
		\item $\lim_{R\to \infty} \beta_R(x) = 0$ for $x\in \Rn$,
	\end{enumerate}
	where $C > 0$ is a positive constant. Let us now define the function
	\[
	\Phi(x, y, t) \coloneqq u(x, t) - v(y, t) - 2\eta^{-1}K(1+|x-y|^2)^{\frac{1}{2}} - \left( \beta_R(x) + \beta_R(y)\right) - \tfrac{\gamma}{T-t}.
	\]
	Note that condition $ii)$ implies that there is a constant $r(R)$ such that $\beta_R(x) \ge \frac{3}{2}L|x|$ if $|x| > r(R)$. Moreover by \eqref{eq:comp_rn_assumption}, we obtain for $|x|, |y| > r(R)$ the estimate
	\begin{align*}
	\Phi(x, y, t) &\le L(1+|x|+|y|) - 2\eta^{-1}K - \tfrac{3}{2}L|x| - \tfrac{3}{2}L|y| - \tfrac{\gamma}{T-t}\\ &= L - 2\eta^{-1}K - \tfrac{\gamma}{T-t} - \tfrac{1}{2}L(|x|+|y|).
	\end{align*}
	Hence, the function has to attain its supremum in a compact subset of $\Rn \times \Rn \times I$.  Let $(\hat{x}, \hat{y}, \hat{t})$ be this maximum. If the asserted inequality \eqref{eq:comp_rn_growth} fails to hold, then $\Phi(\hat{x}, \hat{y}, \hat{t}) > 0$ for $R>0$ big enough. Secondly assume the other case, i.e. $\Phi(\hat{x}, \hat{y}, \hat{t}) > 0$, which implies that
	\begin{equation}
	2\eta^{-1}K|\hat{x}-\hat{y}| \le u(\hat{x}, \hat{t}) - v(\hat{y}, \hat{t}) - \tfrac{\gamma}{T-\hat{t}}. \label{eq:comp_rn_bound}
	\end{equation}
	In case $\hat{t} = 0$ then we would get
	\[
	0 < \Phi(\hat{x}, \hat{y}, 0) \le - 2K(1+|\hat{x}-\hat{y}|^2)^{\frac{1}{2}} - \left( \beta_R(\hat{x}) + \beta_R(\hat{y})\right) + \tfrac{\gamma}{T} \le 0,
	\]
	which is a contradiction. Hence, the maximum $(\hat{x}, \hat{y}, \hat{t})$ lies inside $\Rn \times \Rn \times (0, T)$, yielding
	\begin{align*}
	(a, p + D\beta_R(\hat{x}), X + D^2\beta_R(\hat{x})) &\in \mathcal{P}^{2, +} u(\hat{x}, \hat{t}),\\
	(b, p - D\beta_R(\hat{y}), -X - D^2\beta_R(\hat{y})) &\in \mathcal{P}^{2, -} v(\hat{y}, \hat{t}),
	\end{align*}
	with $a = b+ \frac{\gamma}{(T-\hat{t})^2}$, $p = 2\eta^{-1}K \frac{\hat{x}-\hat{y}}{1+ |\hat{x}-\hat{y}|^2}$, and $$X = \frac{2\eta^{-1}K}{1+|\hat{x}-\hat{y}|^2} \operatorname{Id} - 4\eta^{-1}K \frac{\hat{x}-\hat{y}}{1+|\hat{x}-\hat{y}|^2} \otimes \frac{\hat{x}-\hat{y}}{1+|\hat{x}-\hat{y}|^2}.$$
	This implies that one can find $\mu \in \mathcal{S}(a)$, and $\nu \in \mathcal{S}(b)$ such that
	\begin{align}
	F_*(\hat{x}, \hat{t}, u, a, p + D\beta_R(\hat{x}), X + D^2\beta_R(\hat{x})) &\le \mu G(\hat{x}, \hat{t}, u, p + D\beta_R(\hat{x})) \label{eq:comp_rn_subsol}, \\
	F^*(\hat{y}, \hat{t}, v, b, p - D\beta_R(\hat{y}), -X - D^2\beta_R(\hat{y})) &\ge \nu G(\hat{y}, \hat{t}, v, p - D\beta_R(\hat{y})) \label{eq:comp_rn_supersol}.
	\end{align}
	Subtracting \eqref{eq:comp_rn_subsol} from \eqref{eq:comp_rn_supersol} shows that 
	\begin{align*}
	0 \le&~F^*(\hat y,\hat t, v, b, p - D\beta_R(\hat{y}), -X - D^2\beta_R(\hat{y}))
	- F_*(\hat x,\hat t, u, a, p + D\beta_R(\hat{x}), X + D^2\beta_R(\hat{x})) \\
	&- \nu\left( G(\hat y,\hat  t, v, p - D\beta_R(\hat{y}))
	- G(\hat x,\hat  t, u, p + D\beta_R(\hat{x}))\right) 
	+ (\mu - \nu) G(\hat x,\hat  t, u, p + D\beta_R(\hat{x})).
	\end{align*}
	As $b \le a$, we have $\mu \le \nu$ and we can estimate the last term in the right-hand side above by $0$. To treat other terms we use condition \ref{it:F_inc}, \ref{it:F_growth}, \ref{it:G_Lipschitz}, and \ref{it:G_growth}. As the difference of $F_2^* - {F_2}_*$ is locally bounded and $G_2$ is continuous there is $C(p, X, D\beta_R, D^2\beta_R) > 0$ locally bounded such that
	\begin{align*}
	0 \le&~ \eta (v(\hat y, \hat t) - u(\hat x, \hat t)) + C_{F_1} + \mathcal{S}_{\textrm{max}}C_{G_1} + C(p, X, D\beta_R, D^2\beta_R) \\
	&+ (K_F + \mathcal{S}_{\textrm{max}}K_G) |\hat x - \hat y|.
	\end{align*}
	As $p, X$ are bounded and $D\beta_R$, $D^2 \beta_R$ are bounded independently of $R$ we can introduce a constant $C > 0$ which is independent of $R$ and obtain
	\[
	0 \le C + \eta (v(\hat y, \hat t) - u(\hat x, \hat t))  + K |\hat x - \hat y|.
	\]
	Finally, using \eqref{eq:comp_rn_bound}, we see that $u(\hat x, \hat t) - v(\hat y, \hat t)$ is uniformly bounded, i.e., 
	\[
	u(\hat x, \hat t) - v(\hat y, \hat t) \le \tfrac{2C}{\eta}.
	\]
	The bound on $u-v$ implies that
	\[
	\Phi(x, y, t) \le \Phi(\hat{x}, \hat{y}, \hat{t}) \le u(\hat{x}, \hat{t}) - v(\hat{y}, \hat{t}) \le \tfrac{2C}{\eta}.
	\]
	By sending $R \to \infty$ we obtain
	\[
	u(x, t) - v(y, t) - 2K(1+|x-y|^2)^{\frac{1}{2}} - \tfrac{\gamma}{T-t} \le \tfrac{2C}{\eta}
	\]
	and \eqref{eq:comp_rn_growth} is proved.
	
	\emph{Step 2: Comparison Principle}
	
	It is now time to prove the comparison principle. Let us therefore assume that
	\[
	\sup_{\substack{x\in \Rn \\ t\in [0, T)}} \left\{ u(x, t) - v(x, t)\right\} \eqqcolon \delta > 0
	\]
	and define
	\[
	M_{\alpha, \epsilon, \gamma } \coloneqq \sup_{\substack{x, y \in \Rn \\ t\in I}} \left\{ u(x, t) - v(y, t) - \alpha|x-y|^4 - \epsilon(|x|^2+|y|^2) - \tfrac{\gamma}{T-t}\right\}.
	\]
	Due to the growth estimate \eqref{eq:comp_rn_growth}, $M_{\alpha, \epsilon, \gamma}$ is uniformly bounded. Moreover, we have $M_{\alpha, \epsilon, \gamma} > \delta/2$ for $\gamma, \epsilon$ small enough. Using \eqref{eq:comp_rn_growth}, we see that $M_{\alpha, \epsilon, \gamma}$ is attained at some $(\hat{x}, \hat{y}, \hat{t})$ satisfying
	\begin{align} \label{eq:comp_rn_estimate}
	\alpha|\hat x-\hat y|^4 + \epsilon(|\hat x|^2+|\hat y|^2) &\le u(\hat x, \hat t) - v(\hat y, \hat t) - \tfrac{\gamma}{T-\hat t}
	\le 2K\eta^{-1}|\hat{x} - \hat{y}| + C \\
	\nonumber
	&\le \tfrac{\alpha}{4} |\hat{x}-\hat{y}|^4 + \tfrac{3}{4}(\tfrac{2K\eta^{-1}}{\alpha^{1/4}})^{4/3}  + C,
	\end{align}
	for some constant $C = C(K, C_{F_1}, C_{F_2}) > 0$. Hence, the maximum is achieved inside of the domain, i.e. $(\hat{x}, \hat{y}, \hat{t}) \in \Rn \times \Rn \times (0, T)$ and we can again apply the Jensen-Ishii lemma to obtain that
	\begin{align*}
	(a, p+2\epsilon \hat{x}, X + 2\epsilon\operatorname{Id}) &\in \mathcal{P}^{2, +} u(\hat{x}, \hat{t}), \\
	(b, p-2\epsilon \hat{y}, Y-2\epsilon\operatorname{Id}) &\in \mathcal{P}^{2, -} v(\hat{y}, \hat{t}) ,
	\end{align*}
	with $a-b = \tfrac{\gamma}{(T-\hat{t})^2}$, $p \coloneqq 4\alpha |\hat x-\hat y|^2 (\hat x-\hat y)$, and
	\[
	-4||Z|| \left(\begin{array}{cc}\operatorname{Id} & 0 \\ 0 & \operatorname{Id}\end{array}\right) \le \left(\begin{array}{cc}X & 0 \\ 0 & Y\end{array}\right)\le \left(\begin{array}{cc}Z+\frac{1}{2||Z||}Z^2 & -(Z+\frac{1}{2||Z||}Z^2) \\ -(Z+\frac{1}{2||Z||}Z^2 ) & Z+\frac{1}{2||Z||}Z^2 \end{array}\right),\]
	with $Z \coloneqq 4 \alpha |\hat x-\hat y|^2 \operatorname{Id} + 8\alpha (\hat x-\hat y) \otimes (\hat x-\hat y)$. As $u$ is a subsolution and $v$ is a supersolution, we can find $\mu \in \mathcal{S}(a)$ and $\nu \in \mathcal{S}(b)$ such that
	\begin{align}
	F_*(\hat x,\hat t, u, a, p+2\epsilon\hat x, X+ 2\epsilon\operatorname{Id}) - \mu G(\hat x,\hat  t, u, p + 2\epsilon\hat x) \le 0,\label{eq:comp_rn_proof_sub} \\
	F^*(\hat y,\hat t, v, b, p - 2\epsilon\hat y, Y- 2\epsilon\operatorname{Id}) - \nu G(\hat y,\hat  t, v, p- 2\epsilon\hat y) \ge 0,\label{eq:comp_rn_proof_super}.
	\end{align}
	By subtracting \eqref{eq:comp_rn_proof_sub} from \eqref{eq:comp_rn_proof_super}, we obtain
	\begin{align*}
	0 \le&~ F^*(\hat y,\hat t, v, b, p - 2\epsilon \hat y, Y- 2\epsilon\operatorname{Id}) - F_*(\hat x,\hat t, u, a, p + 2\epsilon \hat x, X + 2\epsilon\operatorname{Id}) \\
	&- \nu G(\hat y,\hat t, v, p - 2\epsilon \hat{y}) + \mu G(\hat x,\hat  t,u, p + 2\epsilon \hat{x}).
	\end{align*}
	By adding and subtracting terms one has that
	\begin{align*}
	0 \le&~ F^*(\hat y,\hat t, v, b, p- 2\epsilon \hat y, Y -2\epsilon\operatorname{Id}) - F_*(\hat x,\hat t, v, a, p + 2\epsilon \hat x, X + 2\epsilon\operatorname{Id}) \\
	&+ F_*(\hat x,\hat t, v, a, p + 2\epsilon \hat x, X + 2\epsilon\operatorname{Id})- F_*(\hat x,\hat t, u, a, p + 2\epsilon \hat x, X + 2\epsilon\operatorname{Id}) \\
	&- \nu G(\hat y,\hat  t, v, p- 2\epsilon \hat y)+ \nu G(\hat x,\hat  t, v, p+ 2\epsilon \hat x) \\
	&- \nu G(\hat x,\hat  t, v, p+ 2\epsilon \hat x)+ \nu G(\hat x,\hat  t, u, p+ 2\epsilon \hat x) \\
	&- \nu G(\hat x,\hat  t, u, p+ 2\epsilon \hat x) + \mu G(\hat x,\hat  t, u, p+ 2\epsilon \hat x)\\
	\le&~\eta(v(\hat y, \hat t) - u(\hat x, \hat t)) \\
	&+ F^*(\hat y,\hat t, v, b, p- 2\epsilon \hat y, Y -2\epsilon\operatorname{Id}) - F_*(\hat x,\hat t, v, a, p + 2\epsilon \hat x, X + 2\epsilon\operatorname{Id}) \\
	&- \nu G(\hat y,\hat  t, v, p- 2\epsilon \hat y)+ \nu G(\hat x,\hat  t, v, p+ 2\epsilon \hat x).
	\end{align*}
	In the second inequality we used \ref{it:F_inc}, \ref{it:G_Lipschitz}, \ref{it:S_mon}, along with $u(\hat x, \hat t) > v(\hat y, \hat t)$, $X\le Y$, and $a\ge b$.
	
	To treat the terms in the last inequality, we use \ref{it:F_growth}, \ref{it:G_growth}, and \ref{it:S_bdd} to obtain,
	\begin{align*}
	\eta \tfrac{\delta}{2} \le&~ F_1(\hat y, \hat t, v, b) - F_1(\hat x, \hat t, v, a) +F_2^*(\hat t, p- 2\epsilon \hat y, Y -2\epsilon\operatorname{Id}) - {F_2}_*(\hat t, p + 2\epsilon \hat x, X + 2\epsilon\operatorname{Id}) \\
	&+\mathcal{S}_\textrm{max} | G_1(\hat y, \hat  t, v) - G_1(\hat x,\hat  t, v) | +\mathcal{S}_\textrm{max} | G_2(t, p + 2\epsilon \hat y) - G_2(t, p- 2\epsilon \hat y) | \\
	\le&~\omega_{F_1}(|\hat x - \hat y| + \tfrac{\gamma}{T-\hat t}) + F_2^*(\hat t, p- 2\epsilon \hat y, Y -2\epsilon\operatorname{Id}) - {F_2}_*(\hat t, p + 2\epsilon \hat x, X + 2\epsilon\operatorname{Id}) \\
	&~+\mathcal{S}_\textrm{max}\omega_{G_1}(|\hat x - \hat y|) +\mathcal{S}_\textrm{max} | G_2(t, p + 2\epsilon \hat y) - G_2(t, p- 2\epsilon \hat y) |
	\end{align*}
	As equation \eqref{eq:comp_rn_estimate} implies that $\alpha |\hat x - \hat y|^4$ is bounded independently of $\epsilon$, and therefore also $p, q, X$, and $Y$, we can take the limit superior as $\epsilon \to 0$ of the above equation to obtain the inequality
	\begin{align*}
	\eta \tfrac{\delta}{2} \le&~\omega_{F_1}(|\hat x - \hat y| + \tfrac{\gamma}{T-\hat t}) + F_2^*(\hat t, p, Y) - {F_2}_*(\hat t, p, X)+\mathcal{S}_\textrm{max}\omega_{G_1}(|\hat x - \hat y|).
	\end{align*}
	Note that we used the upper-semicontinuity of $F_2^* -{F_2}_*$ and the continuity of $G$. Finally, we can again use \ref{it:F_growth} to reach a contradiction as follows
	\begin{align*}
	\eta \tfrac{\delta}{2} &\le\liminf_{\alpha \to \infty}\liminf_{\gamma \to 0}\omega_{F_1}(|\hat x - \hat y| + \tfrac{\gamma}{T-\hat t}) +\omega_{F_2}(|\hat x - \hat y| + \alpha |\hat x - \hat y|^4)+\mathcal{S}_\textrm{max}\omega_{G_1}(|\hat x - \hat y|)\\ &= 0,
	\end{align*}
	if $\liminf_{\gamma \to 0} \liminf_{\epsilon \to 0} \tfrac{\gamma}{T-\hat{t}} = 0$ and $\liminf_{\alpha\to \infty}\liminf_{\gamma \to 0} \liminf_{\epsilon \to 0} \alpha |\hat x - \hat y|^4 = 0$. These equalities hold with a similiar argument as in \cite{Giga1991}.
\end{proof}

\subsection{Stability Result}
The proof is based on \cite[Chapter~6]{Barles2013} with the necessary adaptions.

\begin{proof}[Proof of Theorem \ref{thm:stability}]\label{proof:stability}
	Consider any $(x, t) \in \Omega \times I$ and $(a, p, X) \in \mathcal{P}^{2, +} \overline{u}(x, t)$. By applying Lemma 6.1 from \cite{Barles2013} (see also \cite[Proposition~4.3]{Crandall1992}), we see that there is a sequence $(x_{n_k}, t_{n_k}) \in \Omega \times I$ and $(a_{n_k}, p_{n_k},X_{n_k}) \in \mathcal{P}^{2, +} u_{n_k}(x_{n_k}, t_{n_k})$ such that
	\[
	(x_{n_k}, t_{n_k}, u_{n_k}(x_{n_k}, t_{n_k}), a_{n_k}, p_{n_k},X_{n_k}) \to (x, t, \overline{u}(x, t), a, p, X).
	\]
	As $u_{n_k}$ are viscosity subsolutions, there are $\mu_{n_k} \in \mathcal{S}_{n_k}(a_{n_k})$ such that
	\[
	F_{n_k}(x_{n_k}, t_{n_k}, u_{n_k}, a_{n_k}, p_{n_k}, X_{n_k}) - \mu_{n_k} G_{n_k}(x_{n_k}, t_{n_k}, u_{n_k}, p_{n_k}) \le 0.
	\]
	By the assumptions in the theorem, there is a further subsequence (not relabeled) such that $\mu_{n_k} \to \mu \in \mathcal{S}(a)$. Hence, we have by the definition of the half-relaxed limit $\underline{F}$ and the uniform convergence of $G_{n_k}$ that
	\begin{align*}
	&\underline{F}(x, t, u, a, p, X) - \mu G(x, t, u, p) \\
	&\le \liminf_{k\to \infty } \left(F_{n_k}(x_{n_k}, t_{n_k}, u_{n_k}, a_{n_k}, p_{n_k}, X_{n_k}) - \mu_{n_k} G_{n_k}(x_{n_k}, t_{n_k}, u_{n_k}, p_{n_k})\right) \\
	&\le 0.
	\end{align*}
\end{proof}

\bibliographystyle{plain}

\end{document}